\documentclass[a4paper,UKenglish,cleveref, autoref]{lipics-v2019}

\usepackage{amsmath}
\usepackage{amsthm}
\usepackage{amsfonts}
\usepackage{amssymb}
\usepackage{mathtools}
\usepackage{stmaryrd}
\usepackage{wasysym}
\usepackage{dsfont}
\usepackage{graphicx}
\usepackage{todonotes}
\usepackage{subcaption}
\usepackage{bm}
\usepackage{pst-node}
\usepackage{tikz-cd} 
\usepackage{cleveref}
\newcommand{\R}{\mathbb{R}}

\newcommand{\Z}{\mathbb{Z}}
\newcommand\indep{\perp\!\!\!\perp}
\newcommand{\eps}{\varepsilon}

\newcommand{\exc}{\mathrm{exc}}
\newcommand{\Walks}{\mathfrak W}
\newcommand{\Coals}{\mathfrak C}
\newcommand{\Perms}{\mathfrak S}
\newcommand{\Steps}{A}
\renewcommand{\P}{\Prob}

\DeclareMathOperator{\idf}{\mathds{1}}
\DeclareMathOperator{\Prob}{\mathbb{P}}

\DeclareMathOperator{\pat}{pat}

\DeclareMathOperator{\Leb}{Leb}

\DeclareMathOperator{\Id}{Id}

\DeclareMathOperator{\wcp}{WC}

\DeclareMathOperator{\tree}{Tr}

\DeclareMathOperator{\cpbp}{CP}
\DeclareMathOperator{\bow}{OW}
\DeclareMathOperator{\bobp}{OP}
\DeclareMathOperator{\Perm}{Perm}

\newcommand{\Maps}{\mathfrak m}
\newcommand{\conti}[1]{{\bm{\mathscr #1}}}

\newtheorem{observation}[theorem]{Observation}

\title{Scaling and local limits of Baxter permutations through coalescent-walk processes} 

\titlerunning{Scaling and local limits of Baxter permutations}

\author{Jacopo Borga}{Institut für Mathematik,
	Universität Zürich, Switzerland \and \url{http://www.jacopoborga.com} }{jacopo.borga@math.uzh.ch}{https://orcid.org/0000-0002-2805-7928}{}

\author{Mickaël Maazoun}{Université de Lyon, ENS de Lyon, Unité de mathématiques pures et appliquées, France \and \url{http://perso.ens-lyon.fr/mickael.maazoun/}}{mickael.maazoun@ens-lyon.fr}{https://orcid.org/0000-0001-8852-2345}{}

\authorrunning{Jacopo Borga and Mickaël Maazoun}

\Copyright{Jacopo Borga and Mickaël Maazoun}



\category{}

\relatedversion{A full version of this extended abstract will be submitted later to another journal.}

\supplement{}

\acknowledgements{Thanks to Mathilde Bouvel, Valentin Féray and Grégory Miermont for their dedicated supervision and enlightening discussions. Thanks to Nicolas Bonichon, Emmanuel Jacob, Jason Miller, Kilian Raschel, Olivier Raymond, Vitali Wachtel, for enriching discussions and pointers.}

\EventEditors{Michael Drmota and Clemens Heuberger}
\EventNoEds{2}
\EventLongTitle{31st International Conference on Probabilistic, Combinatorial and Asymptotic Methods for the Analysis of Algorithms (AofA 2020)}
\EventShortTitle{AofA 2020}
\EventAcronym{AofA}
\EventYear{2020}
\EventDate{June 15--19, 2020}
\EventLocation{Klagenfurt, Austria}
\EventLogo{}
\SeriesVolume{159}
\ArticleNo{18}

\begin{document}

\maketitle

\begin{abstract}
Baxter permutations, plane bipolar orientations, and a specific family of walks in the non-negative quadrant are well-known to be related to each other through several bijections. We introduce a further new family of discrete objects, called \emph{coalescent-walk processes}, that are fundamental for our results. We relate these new objects with the other previously mentioned families introducing some new bijections. 

We prove joint Benjamini--Schramm convergence (both in the annealed and quenched sense) for uniform objects in the four families. Furthermore, we explicitly construct a new fractal random measure of the unit square,
called the \emph{coalescent Baxter permuton} and we show that it is the scaling limit (in the permuton sense) of uniform Baxter permutations. 

To prove the latter result, we study the scaling limit of the associated random coalescent-walk processes. We show that they converge in law to a \emph{continuous random coalescent-walk process} encoded by a perturbed version of the Tanaka stochastic differential equation. This result has
connections (to be explored in future projects) with the results of 
Gwynne, Holden, Sun (2016) on scaling limits (in the Peanosphere topology) of plane bipolar triangulations. 

We further prove some results that relate the limiting objects of the four families to each other, both in the local and scaling limit case.
\end{abstract}

\begin{figure}[htbp]
	\begin{minipage}[c]{0.7\textwidth}
		\centering
		\includegraphics[scale=0.35]{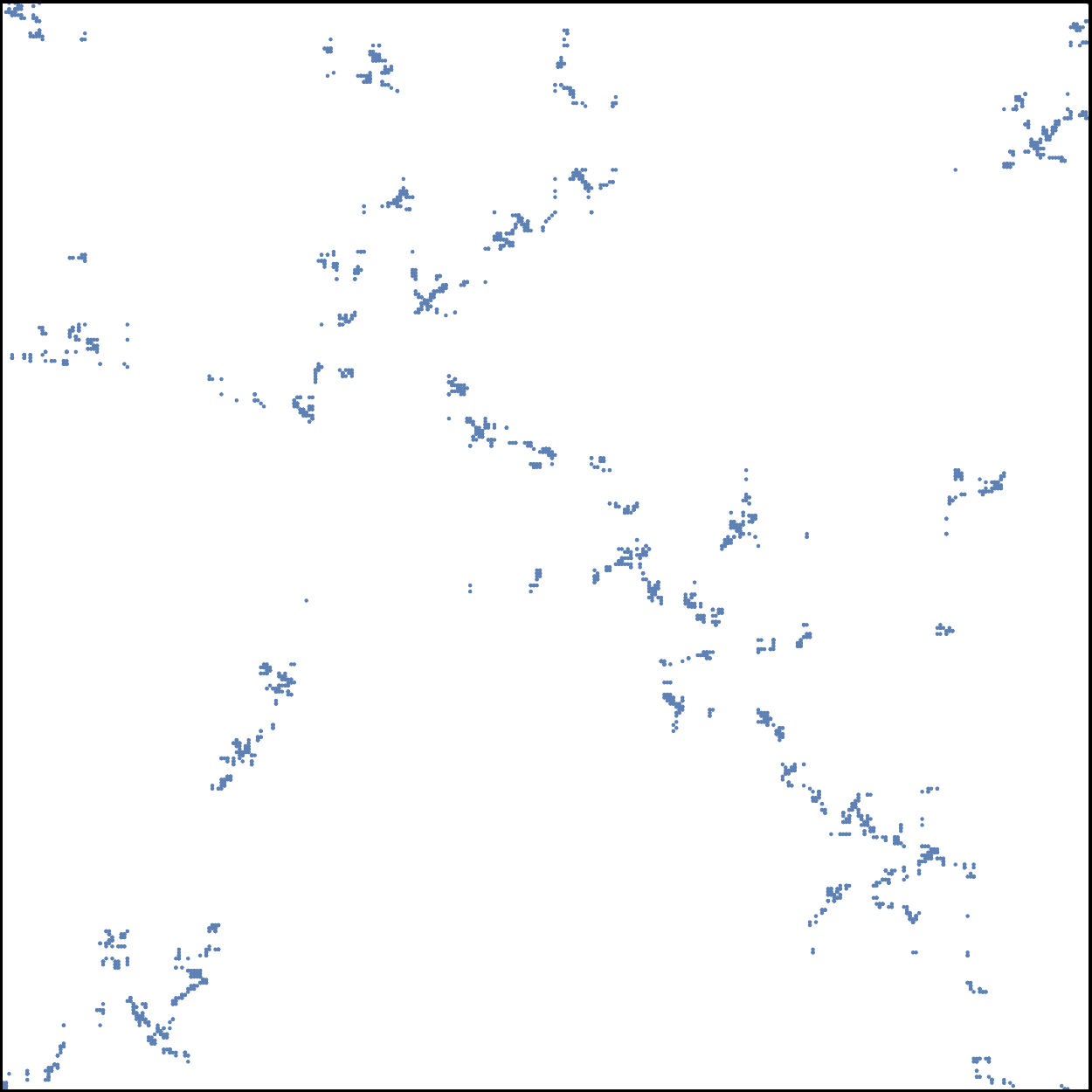}
		\includegraphics[scale=0.35]{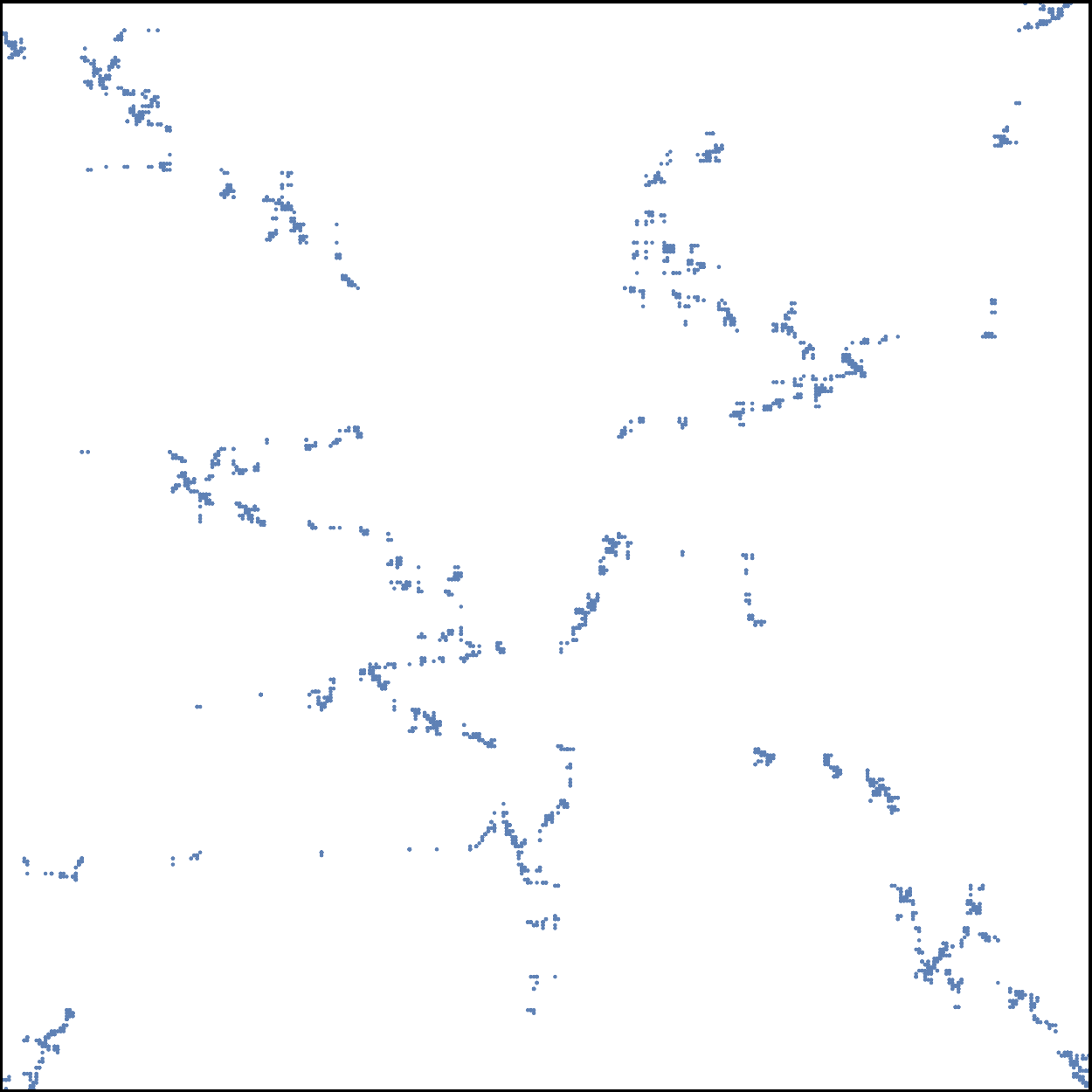}
	\end{minipage}
	\begin{minipage}[c]{0.29\textwidth}
		\caption{The diagrams of two uniform Baxter permutations of size 3253 (left) and 4520 (right). (How these permutations were obtained is discussed in \cref{sect:simulations}).
			\label{fig:Baxter_perm_3253}}
	\end{minipage}
\end{figure}

\section{Introduction and main results}
Baxter permutations were introduced by Glen Baxter in 1964 \cite{MR0184217} to study fixed points of commuting functions.
\emph{Baxter permutations} are permutations avoiding the two vincular patterns $2\underbracket[.5pt][1pt]{41}3$ and $3\underbracket[.5pt][1pt]{14}2$, i.e.\ permutations $\sigma$ such that there are no indices $i < j < k$ such that $\sigma(j+1) < \sigma(i) < \sigma(k) < \sigma(j)$ or $\sigma(j) < \sigma(k) < \sigma(i) < \sigma(j+1)$.

In the last 30 years, several bijections between Baxter permutations, plane bipolar orientations and certain walks in the plane\footnote{We refer to \cref{sect:discrete_objects} for a precise definition of all these objects.} have been discovered. These relations between discrete objects of different nature are a \emph{beautiful piece of combinatorics}\footnote{Quoting the abstract of \cite{MR2763051}.} that we aim at investigating from a more probabilistic point of view in this extended abstract.
The goal of our work is to explore local and scaling limits of these objects and to study the relations between their limits. Indeed, since these objects are related by several bijections at the discrete level, we expect that most of the relations among them also hold in the ``limiting discrete and continuous worlds''. 

We mention that some limits of these objects (and related ones) were previously investigated. 
Dokos and Pak \cite{MR3238333} explored the expected limit shape of doubly alternating Baxter
permutations, i.e.\ Baxter permutations $\sigma$ such that $\sigma$ and $\sigma^{-1}$ are alternating. In their article they claimed that \emph{``it would be interesting to compute the limit shape of random Baxter permutations''}. One of the main goals of our work is to answer this question by proving permuton convergence for uniform Baxter permutations (see \cref{thm:baxter_permuton_conv} below).
For plane walks (i.e.\ walks in $\Z^2$) conditioned to stay in a cone, we mention the remarkable works of Denisov and Wachtel~\cite{MR3342657} and Duraj and Wachtel~\cite{duraj2015invariance} where they proved (together with many other results) convergence towards Brownian meanders or excursions in cones.
This allowed Kenyon, Miller, Sheffield and
Wilson~\cite{MR3945746} to show that the quadrant walks encoding uniformly random plane bipolar orientations (see \cref{sect:KMSW} for more details) converge to a Brownian excursion of correlation $-1/2$ in the quarter-plane. This is interpreted as Peanosphere convergence of the maps decorated by the \textit{Peano curve} (see \cref{sect:KMSW} for further details) to a $\sqrt{4/3}$-Liouville Quantum Gravity (LQG) surface decorated by an independent $\text{SLE}_{12}$. This result was then significantly strengthened by Gwynne, Holden and Sun~\cite{GHS} who proved joint convergence for the map and its dual, in the setting of infinite-volume triangulations.
In proving \cref{thm:baxter_permuton_conv} we extend some of the methods and results of \cite{GHS}, with a key difference in the way limiting objects are defined. We discuss this in more precise terms at the end of this introduction.

So far we have considered three families of objects: Baxter permutations (denoted by $\mathcal{P}$); walks in the non-negative quadrant $(\mathcal W)$ starting on the $y$-axis and ending on the $x$-axis, with some specific admissible increments defined in the forthcoming \cref{eq:admis_steps}; and plane bipolar orientations $(\mathcal{O})$. For our purposes, specifically for the proof of the permuton convergence, we introduce in \cref{sect:discrete_coal_proc} a fourth family of objects called \emph{coalescent-walk processes} $(\mathcal{C})$.

We denote by $\mathcal W_n$ the subset of $\mathcal W$ consisting of quadrant walks of size $n$ (and similarly $\mathcal C_n,\mathcal P_n,\mathcal{O}_n$ for the other three families).
We will present four size-preserving bijections
(denoted using two letters that refer to the domain and co-domain)
between these four families, summarized in the following diagram:
\begin{equation}
\label{eq:comm_diagram}
\begin{tikzcd}
\mathcal{W} \arrow{r}{\wcp}  & \mathcal{C} \arrow{d}{\cpbp} \\
\mathcal{O} \arrow{u}{\bow} \arrow{r}{\bobp}& \mathcal{P}
\end{tikzcd} \;,
\end{equation}
where the mapping $\bow$ was introduced in \cite{MR3945746} and $\bobp$ in \cite{MR2734180}; the others are new.
Our first result is the following:
\begin{theorem} \label{thm:keyres}
	The diagram in \cref{eq:comm_diagram} commutes.
	In particular, $\cpbp\circ\wcp:\mathcal{W}\to\mathcal{P}$ is a size-preserving bijection.
\end{theorem}

Our second result deals with local limits, more precisely Benjamini--Schramm limits. Informally, Benjamini--Schramm convergence for discrete objects looks at the convergence of the neighborhoods (of any fixed size) of a uniformly distinguished point of the object (called root).
In order to properly define the Benjamini--Schramm convergence for the four families, we need to present the respective local topologies. We defer this task to the complete version of this abstract, here we just mention that the local topology for graphs (and so plane bipolar orientations) was introduced by Benjamini and Schramm~\cite{MR1873300} while the local topology for permutations was introduced by the first author~\cite{borga2018local}. Local topologies for plane walks and coalescent-walk processes can be defined in a similar way. 
We denote by $\widetilde \Walks_\bullet$  the completion of the space of rooted walks $\bigsqcup_{n\geq 1} \mathcal W_n \times [n]$  with respect to the metric defining the local topology. The spaces $\widetilde \Coals_\bullet, \widetilde \Perms_\bullet, \widetilde \Maps_\bullet$ are defined likewise from $\mathcal C,\mathcal P, \mathcal O$.

We define below the candidate limiting objects. As a matter of fact, a formal definition requires an extension of the mappings in \cref{eq:comm_diagram} to infinite-volume objects (for the mappings $\wcp$ and $\bow^{-1}$ also an extension to walks that are \emph{not} conditioned in the quadrant). We do not present all the details of such extensions, but they can be easily guessed from our description of the mappings $\wcp,\bow,\cpbp$ and $\bobp$ given in \cref{sect:discrete_objects}.

Let $\nu$ denote the probability distribution on $\Z^2$ given by:
\begin{equation}\label{eq:walk_distrib}
\nu = \frac 12 \delta_{(+1,-1)} + \sum_{i,j\geq 0} 2^{-i-j-3}\delta_{(-i,j)},\quad \text{where $\delta$ denotes the Dirac measure},
\end{equation}
and let\footnote{Here and throughout the paper we denote random quantities using \textbf{bold} characters.} $\bar{\bm W} = (\bar{\bm X},\bar{\bm Y}) = (\bar{\bm W}_t)_{t\in \Z}$ be a bidirectional random plane walk with step distribution $\nu$, with value $(0,0)$ at time 0.
Let $\bar{\bm Z} = \wcp(\bar{\bm W})$ be the corresponding infinite coalescent-walk process, $\bar{\bm \sigma} = \cpbp(\bar{\bm Z})$ the corresponding infinite permutation on $\Z$ (in this context, an infinite permutation is a total order of $\Z$), and $\bar{\bm m} = \bow^{-1}(\bar{\bm W})$ the corresponding infinite map.

\begin{theorem}\label{thm:local} For every $n\in \Z_{>0}$, let $\bm W_n$, $\bm Z_n$, $\bm \sigma_n$, and $\bm m_n$ denote uniform objects of size $n$ in $\mathcal W_n$, $\mathcal C_n$, $\mathcal P_n$, and $\mathcal O_n$ respectively, related by the bijections of \cref{eq:comm_diagram}.
For every $n\in\Z_{>0}$, let $\bm i_n$ be an independently chosen uniform index of $[n]$. Then we have joint convergence in distribution in the space $\widetilde \Walks_\bullet\times\widetilde \Coals_\bullet\times\widetilde \Perms_\bullet\times\widetilde \Maps_\bullet$:
		\[((\bm W_n, \bm i_n), (\bm Z_n, \bm i_n), (\bm \sigma_n, \bm i_n), (\bm m_n, \bm i_n)) \xrightarrow[n\to\infty]{d} (\bar{\bm W},\bar{\bm Z},\bar{\bm \sigma},\bar{\bm m}).\]
\end{theorem}

\begin{remark}
	We give a few comments on this result.
	\begin{enumerate}
		\item The mapping $\bow^{-1}$ naturally endows the map $\bm m_n$ with an edge labeling and the root $\bm i_n$ of $\bm m_n$ is chosen according to this labeling.
		\item We can also prove a quenched version of the above result (of annealed type) for all the four objects (not presented in this extended abstract). It entails (see \cite[Theorem 2.32]{borga2018local}) that consecutive pattern densities of $\bm \sigma_n$ jointly converge in distribution.
		\item The fact that the four convergences are joint follows from the fact that the extensions of the mappings in \cref{eq:comm_diagram} to infinite-volume objects are a.s.\ continuous.
		\item The annealed Benjamini-Schramm convergence for  bipolar orientations to the so-called \textit{Uniform Infinite Bipolar Map} was already proved in \cite[Prop. 3.10]{gwynne2017mating}.
	\end{enumerate}
\end{remark}

Our third (and main) result is a scaling limit result for Baxter permutations (see \cref{fig:Baxter_perm_3253} for some simulations), in the framework of permutons developed by \cite{hoppen2013limits}.
A \emph{permuton} $\mu$ is a Borel probability measure on the unit square $[0,1]^2$ with uniform marginals, that is 
$\mu( [0,1] \times [a,b] ) = \mu( [a,b] \times [0,1] ) = b-a,$
for all $0 \le a \le b\le 1$. Any permutation $\sigma$ of size $n \ge 1$ may be interpreted as the permuton $\mu_\sigma$ given by the sum of Lebesgue area measures
\begin{equation}
\label{eq:perdef}
\mu_\sigma(A)= n \sum_{i=1}^n \Leb\big([(i-1)/n, i/n]\times[(\sigma(i)-1)/n,\sigma(i)/n]\cap A\big),
\end{equation}
for all Borel measurable sets $A$ of $[0,1]^2$.
Let $\mathcal M$ be the set of permutons. As for general probability measure, we say that a sequence of (deterministic) permutons $(\mu_n)_n$ converges \emph{weakly} to $\mu$ (simply denoted $\mu_n \to \mu$) if 
$
\int_{[0,1]^2} f d\mu_n \to \int_{[0,1]^2} f d\mu,
$
for every (bounded and) continuous function $f: [0,1]^2 \to \mathbb{R}$. With this topology, $\mathcal M$ is compact.
Convergence for random permutations is defined as follows:

\begin{definition}\label{defn:perm_conv}
	We say that a random permutation $\bm{\sigma}_n$ converges in distribution to a random permuton $\bm{\mu}$ as $n \to \infty$ if the random permuton $\mu_{\bm{\sigma}_n}$ converges in distribution to $\bm{\mu}$ with respect to the weak topology.  
\end{definition}
Random permuton convergence entails joint convergence in distribution of all (classical) pattern densities (see \cite[Theorem 2.5]{bassino2017universal}). The study of permuton limits, as well as other scaling limits of permutations, is a rapidly developing field in discrete probability theory, see for instance \cite{bassino2017universal, bassino2018brownian,  borga2018localsubclose, borga2019square, hoffman2019scaling, kenyon2015permutations, maazoun, madras2010random, mp}. Our main result is the following:

\begin{theorem}
	\label{thm:baxter_permuton_conv}
	Let $\bm \sigma_n$ be a uniform Baxter permutation of size $n$. There exists a random permuton $\bm{\mu}_{B}$ such that
	$\mu_{\bm \sigma_n}\stackrel{d}{\longrightarrow} \bm{\mu}_{B}.$
\end{theorem}

An explicit construction of the limiting permuton $\bm{\mu}_{B}$, called the \emph{coalescent Baxter permuton}, is given in \cref{sect: constr_limiting_object}. The proof of \cref{thm:baxter_permuton_conv} is based on a result on scaling limits of the coalescent-walk processes $\bm Z_n$, which appears to be of independent interest, and is discussed in \cref{sect:coal_conv}. 
In particular, the convergence of uniform Baxter permutations is joint\footnote{We leave a proper claim of joint convergence to the full version of this paper. However the joint distribution of the scaling limits is the one presented in \cref{sect: constr_limiting_object}.} with that of the conditioned versions of $\bm W_n$ and $\bm Z_n$ presented in \cref{thm:discret_coal_conv_to_continuous}.

\medskip

We finally discuss the relations with the work of Gwynne, Holden and Sun \cite{GHS}. They show that for infinite-volume bipolar oriented triangulations, the explorations of the two tree/dual tree pairs of the map and its dual converge jointly. The limit is the pair of planar Brownian motions which encode the same $\sqrt{4/3}$-LQG surface decorated by both an $\text{SLE}_{12}$ curve and the ``dual'' $\text{SLE}_{12}$ curve, traveling in a direction perpendicular (in the sense of imaginary geometry) to the original curve. As shown below (\cref{lem:equiv_bij}), the bijection of \cite{MR2734180} between plane bipolar orientations and Baxter permutations can be rewritten in terms of the interaction of these two tree/dual tree pairs, which explains the connection between our work and the one of \cite{GHS}.

We prove \cref{thm:baxter_permuton_conv} by extending some of their constructions to finite-volume general maps, which allows us to provide an analog of their result (that are restricted to triangulations) for general plane bipolar orientations in finite volume, jointly with the convergences above\footnote{Not presented in this extended abstract.}. 
More precisely, the coalescent-walk process defined in \cref{sect: bij_walk_coal} is an extension of the random walk $\mathcal X$ defined in \cite[Section 2.1]{GHS}. The fact that it encodes the spanning tree of the dual map (\cref{prop:eq_trees}) is a version of \cite[Lemma 2.1]{GHS}, albeit we present it differently.
Our main technical ingredient is the convergence of the coalescent-walk process driven by a random plane walk of \cref{thm:coal_con_uncond}. It corresponds to \cite[Theorem 4.1] {GHS}.
The way the limiting object (the right-hand side of \cref{eq:coal_con_uncond}) is defined is however very different, and the proofs differ as a consequence. In our case, it comes from a stochastic differential equation (\cref{eq:flow_SDE}), for which existence and uniqueness are known from the literature \cite{MR3882190,MR3098074}. In their case, it is built using imaginary geometry, and characterized by its excursion decomposition. These are nonetheless two descriptions of the same object, providing an SDE formulation of an intricate imaginary geometry coupling. We wish to explore consequences of this in further works.

\medskip

\textbf{Outline of the extended abstract.} The remainder of the abstract is organized as follows. In \cref{sect:discrete_objects} we present the objects and the mappings involved in the diagram in \cref{eq:comm_diagram}. Moreover, we sketch the proof of \cref{thm:keyres}. \cref{sect:coal_and_perm_conv} is devoted to developing the theory for the proof of \cref{thm:baxter_permuton_conv}. In particular, in \cref{sect:coal_conv} we present the aforementioned results for scaling limits of coalescent-walk processes, and in \cref{sect: constr_limiting_object} we give an explicit construction of the limiting permuton for Baxter permutations. Finally, in \cref{sect:coal_con_uncond} we prove our main technical ingredient (\cref{thm:coal_con_uncond}), and in \cref{sect:perm_conv} we finish the proof of \cref{thm:baxter_permuton_conv}. Note that we leave the proof of \cref{thm:local} out of this abstract.

\section{Bipolar orientations, walks in the non-negative quadrant, Baxter permutations and coalescent-walk processes}\label{sect:discrete_objects}

\subsection{Plane bipolar orientations} 
We recall that a \emph{planar map} is a connected graph embedded in the plane with no edge-crossings, considered up to continuous deformation. A map has vertices, edges, and faces, the latter being the connected components of the plane remaining after deleting the edges. The outer face is unbounded, the inner faces are bounded. 
\begin{definition}
	A \emph{plane bipolar orientation} (or simply \emph{bipolar orientation}) is a planar map with oriented edges such that 
	\begin{itemize}
		\item there are no oriented cycles;
		\item there is exactly one vertex with only outgoing edges (the \textit{source}, denoted $s$), and exactly one vertex with only incoming edges (the \textit{sink}, denoted $s'$); all other vertices, called \textit{non-polar},  have both types of edges;
		\item the source and the sink are both incident to the outer face.
	\end{itemize}
	The size of a bipolar orientation $m$ is its number of edges and will be denoted with $|m|$. 
\end{definition}

Every bipolar orientation can be plotted in the plane in such a way that every edge is oriented from bottom to top (as done for example in \cref{fig:bip_orient}). 
\begin{figure}[htbp]
	\begin{minipage}[c]{0.65\textwidth}
		\centering
		\includegraphics[scale=0.4]{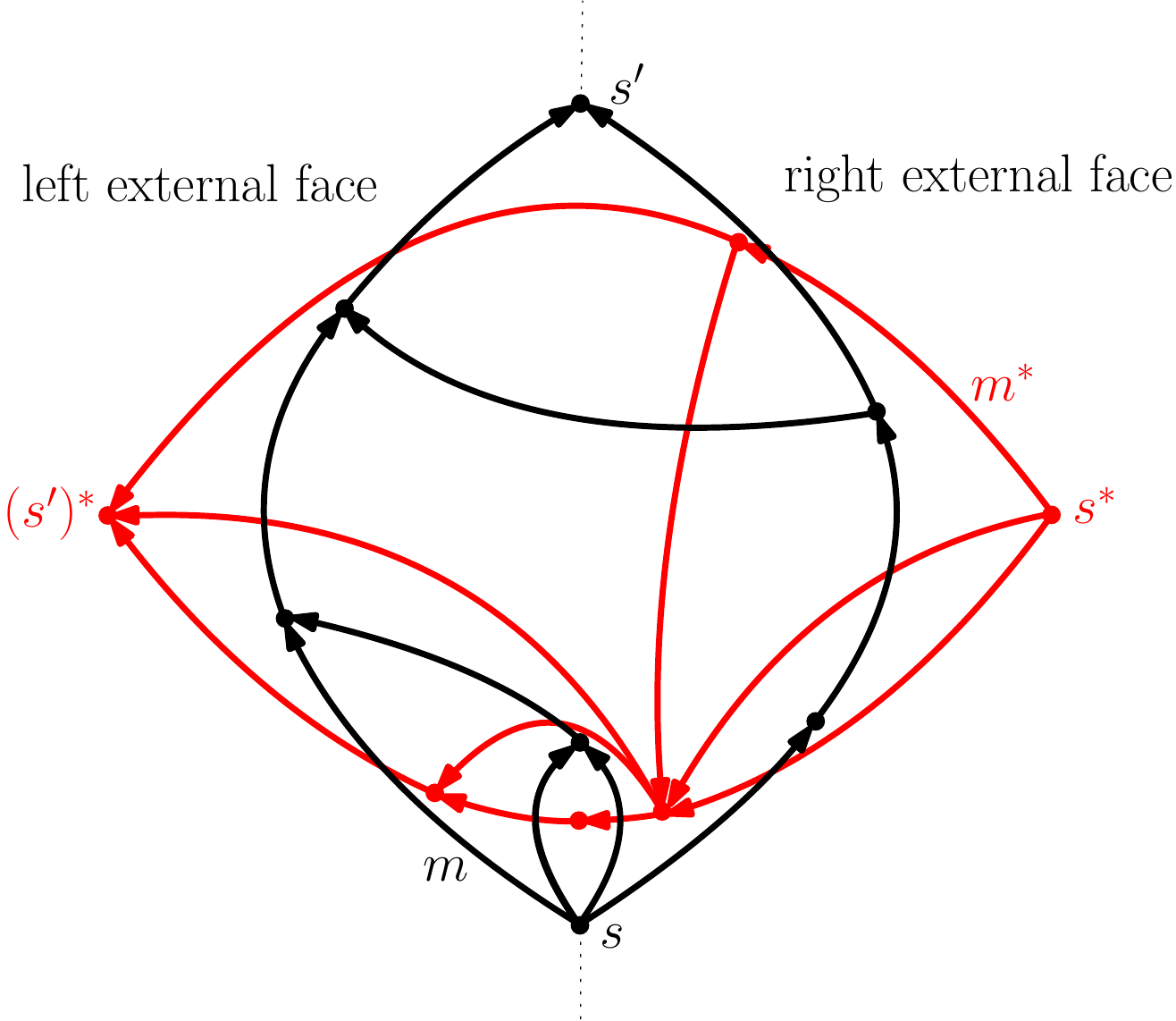}
	\end{minipage}
	\begin{minipage}[c]{0.35\textwidth}
		\caption{In black, a bipolar orientation $m$ of size 10. Note that every edge is oriented from bottom to top. In red, its dual map $m^*$. Similarly, we plot the dual map in such a way that every edge is oriented from right to left. This map will be used in several examples. In later pictures, the orientation of each edge is not displayed. \label{fig:bip_orient}}
	\end{minipage}
\end{figure}

Given a bipolar orientation, an edge $e$ from $v$ to $w$ is bordered, in the clockwise cyclic order, by its bottom vertex, its left face, its top vertex, its right face.
It is useful, for the consistency of definitions, to think of the external face as split in two (see \cref{fig:bip_orient} for an example): the \textit{left external face}, and the \textit{right external face}. 

There is a natural notion of duality for a bipolar orientation $m$. It is the classical duality for (unoriented) maps where the orientation of a dual edge between two primal faces is from right to left. The primal right external face becomes the dual source, and the primal left external face becomes the dual sink. This map $m^*$ is also a bipolar orientation (see~\cref{fig:bip_orient}). The map $m^{**}$ is just the reversal of the map $m$: the source and sink are exchanged, and all edges are reversed. 

Given a bipolar orientation $m$, its \emph{down-right tree} $T(m)$ may be defined as a set of edges equipped with a parent relation, as follows.
\begin{itemize}
	\item The edges of $T(m)$ are the edges of $m$.
	\item Let $e\in m$ and $v$ its bottom vertex.
	\begin{itemize}
		\item If $v$ is the source, then $e$ has no parent edge in $T(m)$ (it is grafted to the root of $T(m)$);
		\item if $v$ is not the source, the parent edge of $e$ in $T(m)$ is the right-most incoming edge of $v$.
	\end{itemize} 
\end{itemize}
The tree $T(m)$ can be drawn on top of $m$: the root of $T(m)$ corresponds to the source $s$ of $m$, internal vertices of $T(m)$ correspond to non-polar vertices of $m$, and leaves of $T(m)$ are the midpoints of some edges of $m$.
Note that one can draw the trees $T(m)$ and $T(m^{**})$ on the map $m$ without any crossing (see the left-hand side of \cref{fig:bip_orient _with_trees} for an example).

We conclude this section recalling that the \emph{exploration} of a tree $T$ is the visit of its vertices (or its edges) following the contour of the tree in the clockwise order.

\begin{figure}[htbp]
	\begin{minipage}[c]{0.66\textwidth}
		\centering
		\includegraphics[scale=0.5]{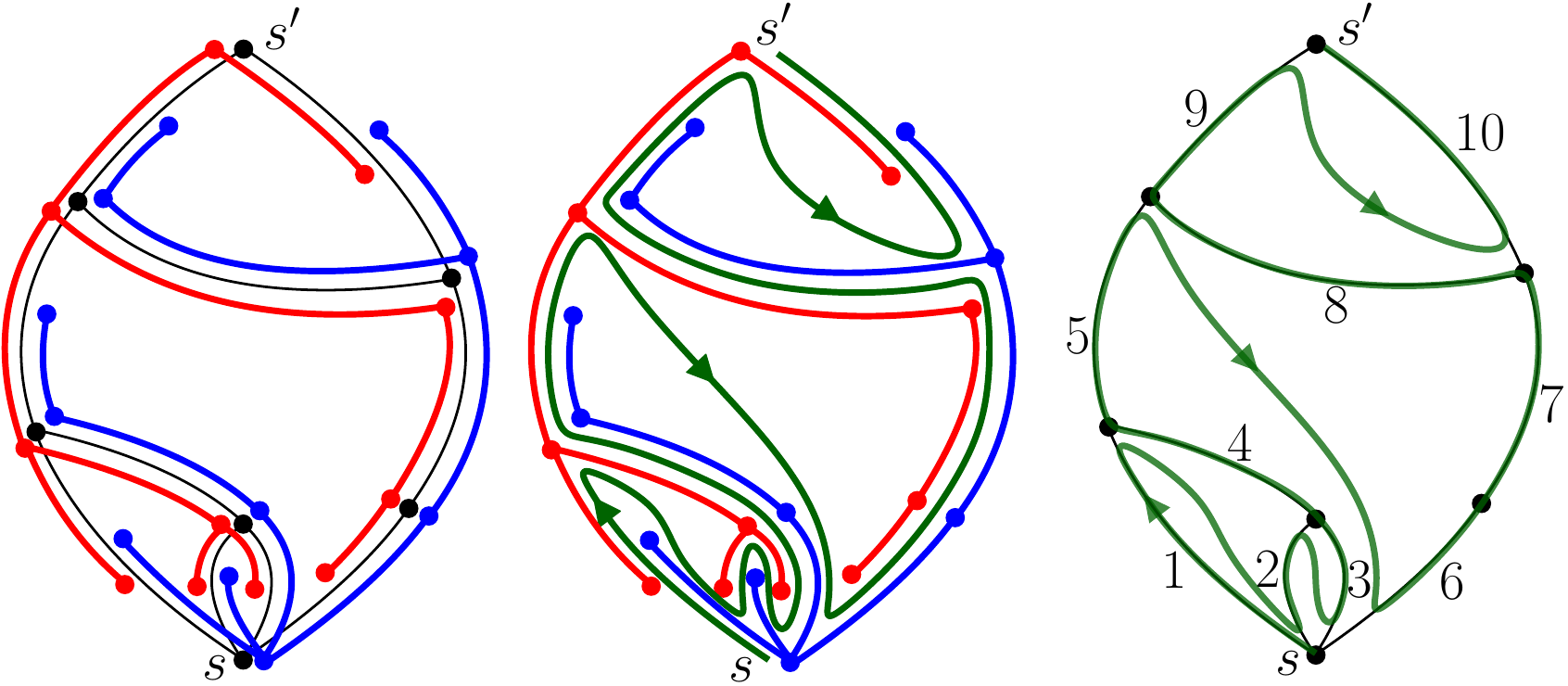}
	\end{minipage}
	\begin{minipage}[c]{0.34\textwidth}
		\caption{\textbf{Left:} A bipolar orientation $m$ with the trees $T(m)$ (in blue) and $T(m^{**})$ (in red). \textbf{Middle:} We add in green the \emph{interface path} tracking the interface between the two trees (see \cref{sect:KMSW}). \textbf{Right:} We label the edges of the bipolar orientations following the interface path. \label{fig:bip_orient _with_trees}}
	\end{minipage}
\end{figure}

\subsection{Kenyon-Miller-Sheffield-Wilson bijection}
\label{sect:KMSW}

We now present a bijection between bipolar orientations and some walks in the non-negative quadrant $\mathbb{Z}_{\geq 0}^2$, introduced in \cite[Section 2]{MR3945746} by Kenyon, Miller, Sheffield and Wilson.

Let $m$ be a bipolar orientation. We consider the exploration of the tree $T(m)$ (highlighted in green in the middle picture of \cref{fig:bip_orient _with_trees}) starting at the source $s$ and ending at the last visit of the sink $s'$. Note that this path (when reversed) is also the exploration of the tree $T(m^{**})$ stopped at the last visit of the source $s$. This path, called \emph{interface path}\footnote{The interface path goes sometimes under the name of \emph{Peano curve}, see for instance \cite{gwynne2019mating}.} since it winds between the trees $T(m)$ and $T(m^{**})$, identifies an ordering on the set $E$ of edges of $m$ since every edge of $T(m)$ corresponds exactly to one edge of $m$ (see the right-hand side of \cref{fig:bip_orient _with_trees} for an example). Let $e_1,e_2,\dots,e_{|m|}$ be the edges of $m$ listed according to this order.

\begin{definition}\label{defn:KMSW_bij}
	Given a bipolar orientation $m$, the corresponding walk $\bow(m)=(W_t)_{t\in [|m|]}=(X_t,Y_t)_{t\in [|m|]}$ of size $|m|$ in the non-negative quadrant $\mathbb{Z}_{\geq 0}^2$ is defined as follows: for $t\in [|m|]$, let $X_t$ be the distance in the tree $T(m)$ between the bottom vertex of $e_t$ and the root of $T(m)$ (corresponding to the source $s$), and let $Y_t$ be the distance in the tree $T(m^{**})$ between the top vertex of $e_t$ and the root of $T(m^{**})$ (corresponding to the sink $s'$).
\end{definition}

\begin{remark}
	The walk $(0,X_1+1,\ldots,X_{|m|}+1)$ is the height process of the tree $T(m)$. The walk $(0,Y_{|m|}+1, Y_{|m|-1}+1,\ldots, Y_1+1)$ is the height process of the tree $T(m^{**})$.
\end{remark}

Suppose that the left external face has $h+1$ edges and the right external face has $k+1$ edges, for some $h,k\geq 0$. Then the walk $(W_t)_{1\leq t\leq |m|}$ starts at $(0,h)$, ends at $(k,0)$, and stays in the non-negative quadrant $\mathbb{Z}_{\geq 0}^2$.
We finally investigate the possible values for the increments of the walk, i.e.\ the values of $W_{t+1}-W_t$. 
We say that two edges of a tree are consecutive if one is the parent of the other.
We first highlight that the interface path of the map $m$ has two different behaviors when following the edges $e_t$ and $e_{t+1}$: 
\begin{itemize}
	\item either it is following two consecutive edges of $T(m)$ (this is the case, for instance, of the edges $e_3$ and $e_4$ on the right-hand side of \cref{fig:bip_orient _with_trees});
	\item or it is first following $e_t$, then it is traversing a face of $m$, and finally is following $e_{t+1}$ (this is the case, for instance, of the edges $e_5$ and $e_6$ on the right-hand side of \cref{fig:bip_orient _with_trees}).
\end{itemize}
When the latter case happens, the interface path splits the boundary of the traversed face in two parts, a left and a right boundary.

Therefore the increments of the walk are either $(+1,-1)$ (when $e_t$ and $e_{t+1}$ are consecutive) or $(-i,+j)$, for some $i,j\in\Z_{\geq 0}$ (when, between $e_t$ and $e_{t+1}$, the interface path is traversing a face with $i+1$ edges on the left boundary and $j+1$ edges on the right boundary). We denote by $\Steps$ the set of possible increments, that is 
\begin{equation}
\label{eq:admis_steps}
\Steps = \{(+1,-1)\} \cup \{(-i,j), i\in \Z_{\geq 0}, j\in \Z_{\geq 0}\}.
\end{equation}
We denote by $\mathcal{W}$ the set of walks in the non-negative quadrant, starting at $(0,h)$ and ending at $(k,0)$ for some $h\geq 0, k\geq 0$, with increments in $\Steps$.

\begin{theorem}(\cite[Theorem 1]{MR3945746})
	The mapping $\bow:\mathcal{O}\to\mathcal{W}$ is a size-preserving bijection.
\end{theorem}

\begin{example}\label{exemp:walk}
	We consider the map $m$ in \cref{fig:bip_orient _with_trees}. The corresponding walk $\bow(m)$ is:
	\begin{equation*}
	\Big((0,2),(0,3),(0,3),(1,2),(2,1),(0,3),(1,2),(2,1),(3,0),(2,0)\Big).
	\end{equation*}
\end{example}

\subsection{Baxter permutations and bipolar orientations}
\label{sect:Baxt_bipol}

In \cite{MR2734180}, a bijection between Baxter permutations and bipolar orientations is given. We give here a slightly different formulation of this bijection (more convenient for our purposes) and then in \cref{lem:equiv_bij} we state that the two formulations are equivalent.

\begin{definition} \label{defn:bobp}
	Let $m$ be a bipolar orientation of size $n\geq 1$. Recall that to every edge of the map $m$ corresponds its dual edge in the dual map $m^*$. The Baxter permutation $\bobp(m)$ associated with $m$ is the only permutation $\pi$ such that for every $1\leq i \leq n$, the $i$-th edge visited in the exploration of $T(m)$ corresponds to the $\pi(i)$-th edge visited in the exploration of $T(m^*)$. We say that this edge \textit{corresponds} to the index $i$.
\end{definition}
An example is given in \cref{fig:bip_orient_and_perm}. The following result proves that $\bobp$ is a bijection.
	
\begin{figure}[htbp]
	\begin{minipage}[c]{0.6\textwidth}
		\centering
		\includegraphics[scale=0.47]{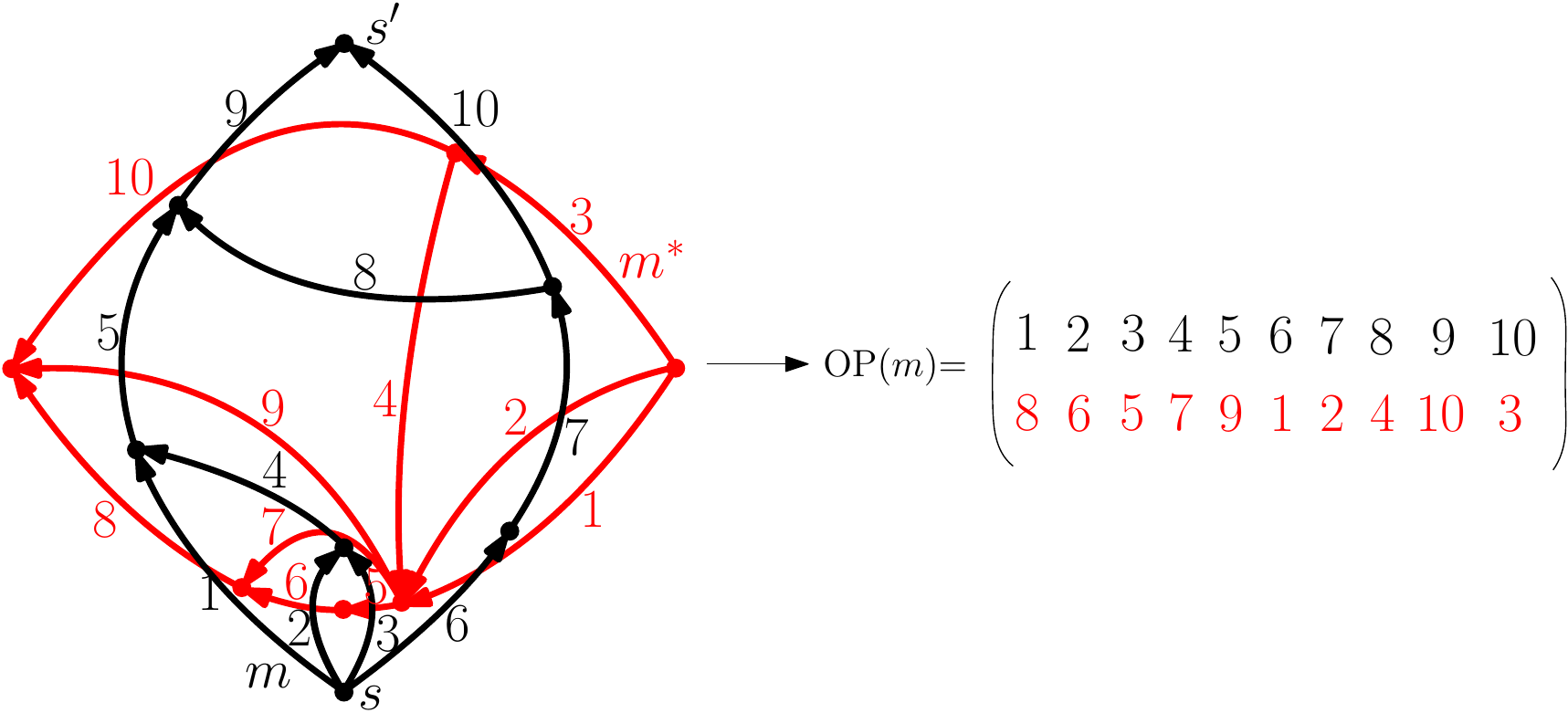}
	\end{minipage}
	\begin{minipage}[c]{0.4\textwidth}
		\caption{\textbf{Left:} The bipolar orientation $m$ and its dual $m^*$, already considered in \cref{fig:bip_orient}. We plot in black the labeling of the edges of $m$ obtained in \cref{fig:bip_orient _with_trees} and in red the labeling of the edges of $m^*$ obtained using the same procedure used for $m$. \textbf{Right:} The permutation $\bobp(m)$ obtained by pairing the labels of the corresponding primal and dual edges between $m$ and $m^*$.  \label{fig:bip_orient_and_perm}}
	\end{minipage}
\end{figure}

\begin{lemma}
	\label{lem:equiv_bij}
	The function $\bobp:\mathcal{O}\to\mathcal{P}$ is equal to the function $\Psi:\mathcal{O}\to\mathcal{P}$ defined in \cite[Section 3.2]{MR2734180}. Therefore $\bobp$ is a size-preserving bijection.
\end{lemma}
The definition of $\Psi$ is the same as that of $\bobp$, with $T(m^*)$ replaced by $T(m^{-1})$, $m^{-1}$ denoting the symmetry of $m$ along the vertical axis. So the proof (that we omit) amounts to showing that these two trees visit the edges of $m$ in the same order\footnote{Actually they are related by a classic bijection between trees: the Lukasiewicz walk of $T(m^*)$ is the reversal of the height function of $T(m^{-1})$.}. 

\subsection{Discrete coalescent-walk processes}\label{sect:discrete_coal_proc}

Since the key ingredient for permuton convergence is the extraction of patterns (see \cref{prop:perm_charact}), we introduce in this section a new tool in order to ``extract patterns from the plane walk'' that encodes a Baxter permutation, namely \emph{coalescent-walk processes}. Then, in \cref{sect: bij_walk_coal}, we present a bijection between walks in the non-negative quadrant and a specific kind of coalescent-walk processes, and in \cref{sect:from_coal_to_perm},  we introduce a bijection between these coalescent-walk processes and Baxter permutations. Composing these two mappings we obtain another bijection between walks in the non-negative quadrant and Baxter permutations. Finally, in \cref{sect:equiv_bij} we complete the proof of \cref{thm:keyres}.

\begin{definition}
	\label{def:discrete_coal_process}
	Let $I$ be a (finite or infinite) interval of $\Z$. We call \emph{coalescent-walk process} over $I$ a family $\{(Z^{(t)}_s)_{s\geq t, s\in I}\}_{t\in I}$ of one-dimensional walks such that
	\begin{itemize}
		\item the walk $Z^{(t)}$ starts at zero at time $t$, i.e. $Z^{(t)}_t = 0$;
		\item if $Z^{(t)}_k\geq Z^{(t')}_k$ (resp. $Z^{(t)}_k\leq Z^{(t')}_k$) at some time $k,$ then $Z^{(t)}_{k'}\geq Z^{(t')}_{k'}$ (resp. $Z^{(t)}_{k'}\leq Z^{(t')}_{k'}$) for every $k'\geq k$.
	\end{itemize}
\end{definition}

Note that, as a consequence, if $Z^{(t)}_k= Z^{(t')}_k$, at time $k,$ then $Z^{(t)}_{k'}=Z^{(t')}_{k'}$ for every $k'\geq k$. In this case, we say that $Z^{(t)}$ and $Z^{(t')}$ are \emph{coalescing} and call \emph{coalescent point} of $Z^{(t)}$ and $Z^{(t')}$ the point $(\ell,Z^{(t)}_\ell)$ such that  $\ell=\min\{k\geq \max\{t,t'\}|Z^{(t)}_{k}=Z^{(t')}_{k}\}$. We denote by $\Coals(I)$ the set of coalescent-walk processes over some interval $I$.

\subsubsection{The coalescent-walk process corresponding to a plane walk}
\label{sect: bij_walk_coal}
We now introduce a particular family of coalescent-walk processes of interest for us.
Let $I$ be a (finite or infinite) interval of $\Z$. Recall the definition of $\Steps$ from \cref{eq:admis_steps} page \pageref{eq:admis_steps}, and let $\Walks_\Steps(I)$ be the set of plane walks of time space $I$ (functions $I\to \Z^2$) with increments in $\Steps$.

Take $W\in\Walks_\Steps(I)$ and denote $W_t = (X_t,Y_t)$ for $t\in I$.
From $X$ and $Y$ we construct the family of walks $\{Z^{(t)}\}_{t\in I}$, called the coalescent-walk process associated with $W,$ by
\begin{itemize}
	\item for $t\in I$, $Z^{(t)}_t=0;$
	\item for $t\in I$ and $k\in I\cap[t+1,+\infty),$
	\begin{equation}\label{eq:increments}
		Z^{(t)}_{k}=\begin{cases}Z^{(t)}_{k-1}+(Y_{k}-Y_{k-1}), &\quad\text{if}\quad Z^{(t)}_{k-1}\geq0,\\
		Z^{(t)}_{k-1}-(X_{k}-X_{k-1}),&\quad\text{if}\quad Z^{(t)}_{k-1}<0\text{ and }Z^{(t)}_{k-1}-(X_{k}-X_{k-1})<0,\\
		Y_{k}-Y_{k-1},&\quad\text{if}\quad Z^{(t)}_{k-1}<0\text{ and }Z^{(t)}_{k-1}-(X_{k}-X_{k-1})\geq0.
		\end{cases}
	\end{equation}
\end{itemize}
 
Let $\wcp: \Walks_\Steps(I) \to \Coals(I)$ map each $W\in \Walks_\Steps(I)$ to the corresponding coalescent-walk process, i.e.\ $\wcp(W)=\{Z^{(t)}\}_{t\in I}$. We also set $\mathcal{C}_n = \wcp(\mathcal W_n) \subset \Coals([n])$ and $\mathcal{C}=\cup_{n\in \Z_{\geq 0}}\mathcal{C}_n.$

\medskip

We give a heuristic explanation of this construction in the following example.
\begin{example}
	We consider the plane walk $W=(W_t)_{t\in [10]}$ starting at $(0,0)$ on the left-hand side of \cref{example_discrete_coal}. We plot in the second diagram of \cref{example_discrete_coal} the walks $Y$ in red and $-X$ in blue. We now explain how we reconstruct the ten walks $\{Z^{(t)}\}_{1\leq t\leq 10}$ (in green on the right-hand side of \cref{example_discrete_coal}). The walk $Z^{(t)}$ starts at height zero at time $t$.
	Then, 
	\begin{itemize}
		\item If $Z^{(t)}_{k-1}$ is non-negative (in particular at the starting point), then the increment $Z^{(t)}_{k}-Z^{(t)}_{k-1}$ is the same as the one of the red walk.
		\item If $Z^{(t)}_{k-1}$ is negative, then the increment $Z^{(t)}_{k}-Z^{(t)}_{k-1}$ is the same as the one of the blue walk, as long as this increment keeps $Z^{(t)}_{k}$ negative.
		\item Now if at time $k-1,$ $Z^{(t)}_{k-1}$ is negative but the blue increment would ``force'' it to cross (or touch) the $x$-axis (that is if $X_{k}-X_{k-1}\leq Z^{(t)}_{k-1}<0$), then $Z^{(t)}_{k}$ is equal to $Y_{k}-Y_{k-1}$ (i.e. $Z^{(t)}$ coalesces with $Z^{(k-1)}$ at time $k$). For instance this is the case of the second increment of the walk $Z^{(7)}$. 
	\end{itemize} 
	
	\begin{figure}[b]
		\begin{minipage}[c]{0.7\textwidth}
			\centering
			\includegraphics[scale=.5]{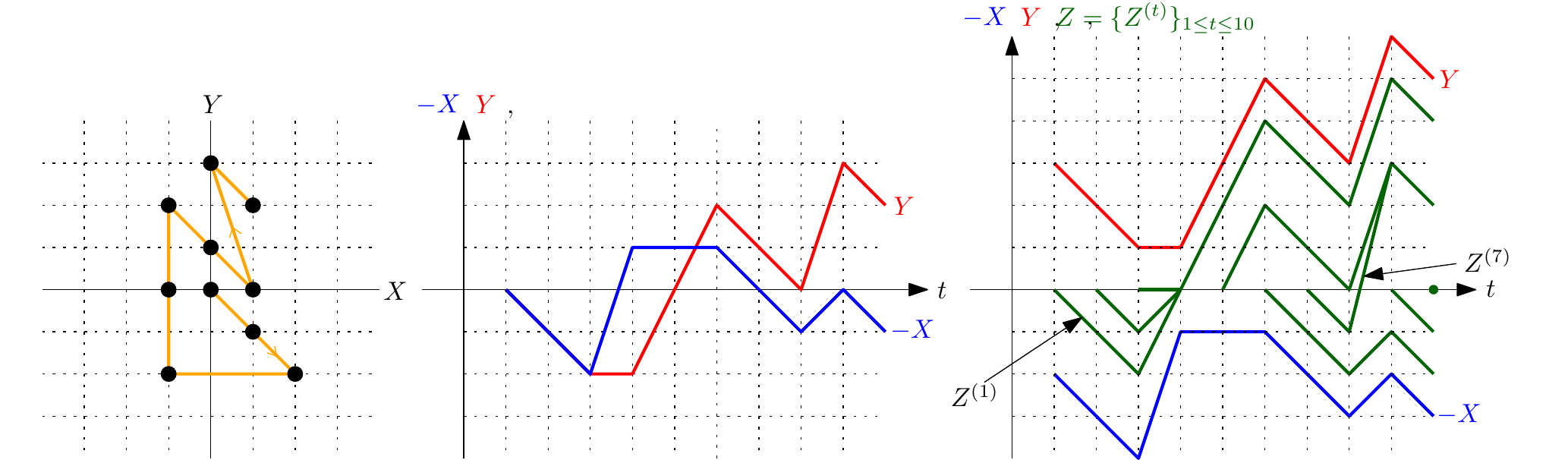}
		\end{minipage}
		\begin{minipage}[c]{0.29\textwidth}
			\caption{\textbf{Left:} A plane walk $(X,Y)$ starting at (0,0). \textbf{Middle:} The diagram of the walks $Y$ (in red) and $-X$ (in blue). \textbf{Right:} The two walks are shifted (one towards the top and one to the bottom) and the ten walks of the coalescent-walk process are plotted in green. \label{example_discrete_coal}}
		\end{minipage}
	\end{figure}
\end{example}

\begin{observation}
	The $y$-coordinates of the coalescent points of a coalescent-walk process in $\mathcal{C}(I)$ are non-negative. 
\end{observation}

\subsubsection{The permutation associated with a coalescent-walk process}\label{sect:from_coal_to_perm}

Given a coalescent-walk process $Z = \{Z^{(t)}\}_{t\in I}$ defined on a (finite or infinite) interval $I$, the relation $\leq_Z$ on $I$ defined as follows is a total order (we skip the proof of this fact): 
\begin{equation}\label{eq:coal_to_perm}
i\leq_Z j\quad\iff \quad
 \{i<j \text { and } Z^{(i)}_j<0\}\quad
\text{or}\quad \{j<i \text { and } Z^{(j)}_i\geq0\}\quad
\text{or}\quad \{i=j\}.
\end{equation}
This definition allows to associate a permutation to a coalescent-walk process.

\begin{definition}Fix $n\in\Z_{\geq 0}$. Let $Z = \{Z^{(t)}\}_{t\in [n]} \in \mathcal C_n$ be a coalescent-walk process over $[n]$. Denote $\cpbp(Z)$ the permutation $\sigma \in \Perms_n$ such that for $1\leq i, j\leq n$, 
	$\sigma(i)\leq\sigma(j) \iff i\leq_Z j$.
\end{definition}

We have that pattern extraction in the permutation $\cpbp(Z)$ depends only on a finite number of trajectories, a key step towards proving permuton convergence.

\begin{proposition}
	\label{cor:patterns}
	Let $\sigma$ be a permutation obtained from a coalescent-walk process $Z=\{Z^{(t)}\}_{t\in[n]}$ via the map $\cpbp$. Let $I=\{i_1<\dots<i_k\} \subset [n]$. Then\footnote{See \cref{sect:perm_conv} for notation on patterns of permutations.} $\pat_I(\sigma)=\pi$ if the following condition holds: for all $1\leq \ell< s \leq k,$ $Z^{(i_\ell)}_{i_s} \geq 0  \iff  \pi(s)<\pi(\ell).$
\end{proposition}

We end this section with the following observation. Note that given a coalescent-walk process on $[n]$, the plane drawing of the trajectories $\{Z^{(t)}\}_{t\in I}$ identifies a natural tree structure $\tree(Z)$ as follows (see for instance the middle and right-hand side of \cref{fig:bip_orient_with_coal}):
\begin{itemize}
	\item vertices of $\tree(Z)$ correspond to points $1,\ldots,n$ on the $x$-axis, plus a root.
	\item Edges are portions of trajectories starting at the right of a vertex $i$ and interrupted at the first encountered new vertex. Trajectories that do not encounter a new vertex before time $n$ are connected to the root. The label $i$ is also carried by the edge at the right of $i$.
\end{itemize}

\begin{remark}\label{rk:cpbp_through_fortree}
	In the case where $I=[n]$ for some $n\in\Z_{\geq 0}$, the permutation $\pi = \cpbp(Z)$ is readily obtained from $\tree(Z)$: it is enough to label the points $1,\dots,n$ on the $x$-axis of the diagram of the colaescent-walk process $Z$ (these labels are painted in purple in the middle picture of \cref{fig:bip_orient_with_coal}) according to the exploration process of $\tree(Z)$ and then to read these labels from left to right.
\end{remark}

\begin{figure}[htbp]
	\centering
	\includegraphics[scale=0.6]{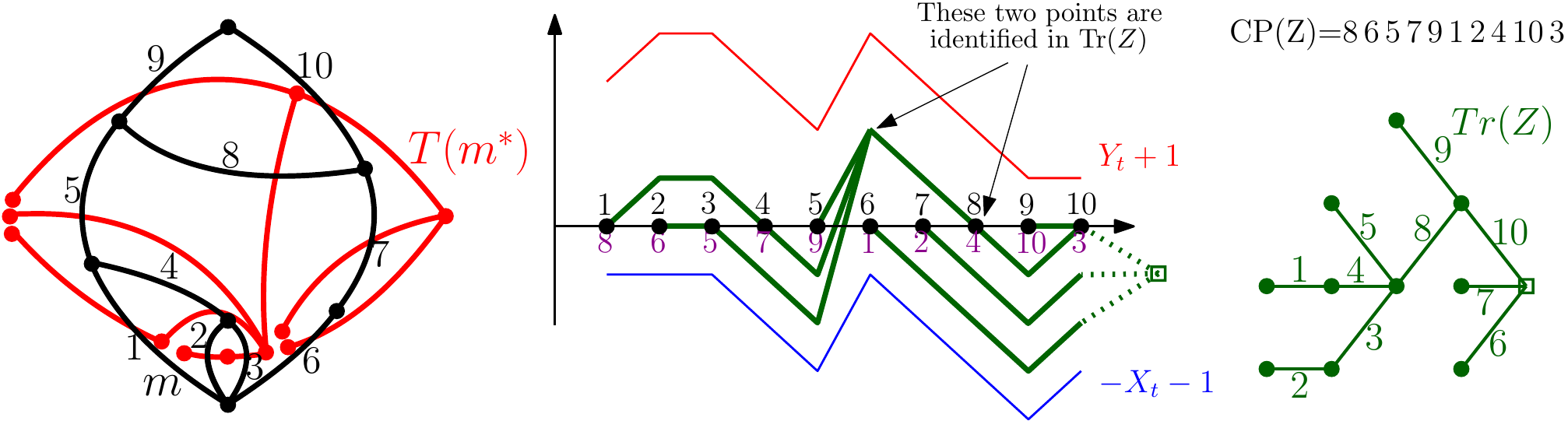}
	\caption{On the left-hand side the map $m$ from \cref{fig:bip_orient_and_perm}. In the middle the associated coalescent-walk process $Z=\wcp\circ\bow(m)$ that naturally determines the tree $\tree(Z)$ (shown on the right). Note that the exploration of $\tree(Z)$ gives the inverse permutation $\cpbp(Z)^{-1}=6\,7\,10\,8\,3\,2\,4\,1\,5\,9$.
	\label{fig:bip_orient_with_coal}}
\end{figure}

\subsection{From plane walks to Baxter permutations via coalescent-walk processes.}
\label{sect:equiv_bij}

We sketch here the proof of \cref{thm:keyres}.
The key ingredient is to show that the dual tree $T(m^*)$ of a bipolar orientation can be recovered from its encoding plane walk by building the associated coalescent-walk process $Z$ and looking at the corresponding tree $\tree(Z)$.
More precisely, let $W = (W_t)_{1\leq t\leq n} = \bow (m)$ be the walk encoding a given bipolar orientation $m$, and $Z = \wcp(W)$ be the corresponding coalescent-walk process. Then the following result, illustrated by an example in \cref{fig:bip_orient_with_coal}, holds.

\begin{proposition}\label{prop:eq_trees} 
		The tree $\tree(Z)$ is equal to the dual tree $T(m^*)$ with edges labeled according to the order given by the exploration of $T(m)$.
\end{proposition}
The proof requires a lot more notation so we skip it in this extended abstract.  \Cref{thm:keyres} then follows immediately, by construction of $\bobp(m)$ from $T(m^*)$ and $T(m)$ (\cref{defn:bobp}) and of $\cpbp(Z)$ from $\tree(Z)$ (\cref{rk:cpbp_through_fortree}).

\section{Convergence to the Baxter permuton}\label{sect:coal_and_perm_conv}
We start this section by representing a uniform random walk in $\mathcal W_n$ as a conditioned random walk.
For all $n\geq 2$, let $\Walks^{A,\exc}_n$ be the set of plane walks $(W_t)_{0\leq t\leq n-1}$ of length $n$ that stay in the non-negative quadrant, starting and ending at $(0,0)$, with increments in $A$ (defined in \cref{eq:admis_steps}). Remark that for $n\geq 1$, the mapping $\mathcal \Walks^{A,\exc}_{n+2} \to \mathcal W_{n}$ removing the first and the last step, i.e.\ $W \mapsto (W_{t})_{1\leq t \leq n}$,
is a bijection. Recall also that $\bar{\bm W}$ denotes the walk defined below \cref{eq:walk_distrib}. An easy calculation then gives the following (observed also in \cite[Remark 2]{MR3945746}):

\begin{proposition}[]\label{prop:unif_law}
	Conditioning on $\{(\bar{\bm W}_t)_{0\leq t\leq {n+1}}  \in \mathcal \Walks^{A,\exc}_{n+2}\}$, the law of  $(\bar{\bm W}_t)_{0\leq t \leq n+1}$ is the uniform distribution on $\Walks^{A,\exc}_{n+2}$, and the law of $(\bar{\bm W}_t)_{1\leq t \leq n}$ is the uniform distribution on $\mathcal W_n$.
\end{proposition}

As we said in the introduction, a key result to prove \cref{thm:baxter_permuton_conv} is to determine the scaling limit of coalescent-walk processes encoded by uniform elements of $\mathcal W_n$. Thanks to \cref{prop:unif_law} we can equivalently study coalescent-walk processes encoded by quadrant walks conditioned to start and end at $(0,0)$. We will first deal with the unconditioned case (see \cref{sect:uncond_case}) and then with the conditioned one (see \cref{sect:condit_case}).

\subsection{Scaling limits of coalescent-walk processes}\label{sect:coal_conv} 
We start by defining our continuous limiting object: it is formed by the solutions of the following family of stochastic differential equations (SDEs) indexed by $u\in \R$, driven by a two-dimensional process $\conti W = (\conti X,\conti Y)$
\begin{equation}\label{eq:flow_SDE}
\begin{cases}
d\conti Z^{(u)}(t) &= \idf_{\{\conti Z^{(u)}(t)> 0\}} d\conti Y(t) - \idf_{\{\conti Z^{(u)}(t)\leq 0\}} d \conti X(t), \quad t\geq u,\\
\conti Z^{(u)}(t)&=0,\quad  t\leq u.
\end{cases} 
\end{equation}
Existence and uniqueness of solutions were already studied in the literature in the case where the driving process $\conti W$ is a Brownian motion, in particular with the following result.
\begin{theorem}[Theorem 2 of \cite{MR3098074}, Proposition 2.2 of \cite{MR3882190}]\label{thm:ext_and_uni}
	\label{thm:uniqueness}
	Let $I$ be a (finite or infinite) interval of $\mathbb R$ and fix $t_0\in I$. Let $\conti W = (\conti X,\conti Y)$ denote a two-dimensional Brownian motion on  $I$ with covariance matrix $\begin{psmallmatrix}1 &\rho \\ \rho &1
	\end{psmallmatrix}$ for $\rho \in (-1,1)$. We have path-wise uniqueness (explained in 1 below) and existence (explained in 2 below) of a strong solution for the SDE:
	\begin{equation} \label{eq:SDE}
	\begin{cases}
	d \conti Z(t) &= \idf_{\{\conti Z(t)> 0\}} d \conti Y(t) - \idf_{\{\conti Z(t)\leq 0\}} d \conti X(t), \quad  t\in I\cap[t_0,+\infty),\\
	\conti Z(t_0)&=0.
	\end{cases} 
	\end{equation}
	Namely, letting $(\Omega, \mathcal F, (\mathcal F_t)_{t\in I}, \Prob)$ be a filtered probability space satisfying the usual conditions, and assuming that $(\conti X,\conti Y)$ is an $(\mathcal F_t)_t$-Brownian motion, 
	\begin{enumerate}
		\item if $\conti Z,\conti Z^\star$ are two $ (\mathcal F_t)_t$-adapted continuous processes that verify \cref{eq:SDE} almost surely, then $\conti Z=\conti Z^\star$ almost surely.
		\item There exists an $(\mathcal F_t)_t$-adapted continuous process $\conti Z$ which verifies \cref{eq:SDE} almost surely, and is adapted to the completion of the canonical filtration of $(\conti X,\conti Y)$.
	\end{enumerate}
\end{theorem}

\subsubsection{The unconditioned scaling limit result}\label{sect:uncond_case}
Let us now work on the completed canonical filtered probability space of a Brownian motion $\conti W = (\conti X, \conti Y)$ with covariance $\begin{psmallmatrix}1 &-1/2 \\ -1/2 &1
\end{psmallmatrix}$. For $u\in \R$, let $\conti Z^{(u)}$ be the strong solution of \cref{eq:SDE} with $I=[u,\infty)$ and $t_0=u$, provided by \cref{thm:ext_and_uni}. Note that $\conti Z^{(u)}$ satisfies \cref{eq:SDE} (only) for almost all $\omega$. For every $u$, $\conti Z^{(u)}$ is adapted, and it is simple to see that the map $(\omega,u)\mapsto \conti Z^{(u)}$ is jointly measurable. By Tonelli's theorem, for almost every $\omega$, $\conti Z^{(u)}$ is a solution for almost every $u$.
\begin{remark}\label{rem:solution_of_SDE_are_BM}
	For fixed $u$, $\conti Z^{(u)}$ is a Brownian motion on $[u,\infty)$. Note however that the coupling of $\conti Z^{(u)}$ for different values of $u$ is highly non trivial.
\end{remark}

\begin{remark}
	Given $\omega$ (even restricted to a set of probability one), we cannot say that $(\conti Z^{(u)})_{u\in \R}$ forms a whole field of solutions to \cref{eq:flow_SDE}, since we cannot guarantee that the SDE holds for all $u$ simultaneously. Similarly, it is expected that there exists exceptional $u$ where uniqueness fails.
\end{remark}

Now, let $\bar{\bm W} = (\bar{\bm X},\bar{\bm Y}) =(\bar{\bm X}_k,\bar{\bm Y}_k)_{k\in \Z}$ be the random plane walk defined below \cref{eq:walk_distrib}, and $\bar{\bm Z} = \wcp(\bar{\bm W})$ be the corresponding coalescent-walk process. 
We define rescaled versions: for all $n\geq 1, u\in \R$, let $\bar {\conti W}_n:\R\to \R^2$ and $\bar {\conti Z}^{(u)}_n:\R\to\R$ be the continuous functions defined by linearly interpolating the following points:
\begin{equation}\label{eq:rescaled_version}
\bar{\conti W}_n\left(\frac kn\right) = \frac 1 {\sqrt {2n}} \bar{\bm W}_{k},\quad k\in \Z, \qquad
\bar{\conti Z}^{(u)}_{n}\left(\frac kn\right) = \frac 1 {\sqrt {2n}} \bar{\bm Z}^{(\lfloor nu\rfloor)}_{k},\quad u\in \R, k\in \Z.
\end{equation}
Our most important technical result is the following theorem (whose proof is postponed to \cref{sect:coal_con_uncond}).
\begin{theorem}\label{thm:coal_con_uncond}
	Let $u_1< \ldots < u_k$. 
	We have the following joint convergence in $(\mathcal C(\R,\R))^{k+2}$: 
	\begin{equation} 
	\left(\bar{\conti W}_n,\bar{\conti Z}^{( u_1)}_n,\ldots,\bar{\conti Z}^{(u_k)}_n\right) 
	\xrightarrow[n\to\infty]{d}
	\left(\conti W,\conti Z^{(u_1)},\ldots,\conti Z^{(u_k)}\right).
	\label{eq:coal_con_uncond}
	\end{equation}
\end{theorem}

\subsubsection{The conditioned scaling limit result}\label{sect:condit_case}

As a standard application of \cite[Theorem 4]{duraj2015invariance}, the scaling limit of the random walk $\bar {\conti W}_n$ conditioned on starting at the origin at time 0 and ending at the origin at time $n+1$ is $\conti W_e=(\conti X_e, \conti Y_e)$, the Brownian excursion in the non-negative quadrant of covariance $\begin{psmallmatrix}1 &-1/2 \\ -1/2 &1
\end{psmallmatrix}$. Let us denote by $(\Omega, \mathcal F, (\mathcal F_t)_{0\leq t \leq 1}, \Prob_\exc)$ the completed canonical probability space of $\conti W_e$, and work from now on in that space.

It makes sense that the scaling limit of the coalescent-walk process in this conditioned setting should be the solution of \cref{eq:flow_SDE} driven by $\conti W_e$, for which we have to show existence and uniqueness. First, let us remark that since Brownian excursions are semimartingales, stochastic integrals are still well-defined. We skip the rather abstract proof of the following, which relies on absolute continuity between Brownian excursion and Brownian motion:

\begin{theorem}\label{thm:ext_and_uni_excursion}
	Denote $\mathcal F^{(u)}_t = \sigma (\conti W_e(s)-\conti W_e(u), u\leq s \leq t)$ completed by the $\Prob_\exc$-negligible sets of $\Omega$. 
	There is a jointly measurable map $(\omega, u)\mapsto \conti Z_e^{(u)}$ such that for all $u$, 
	$\conti Z_e^{(u)}$ is $(\mathcal F^{(u)}_t)_t$-adapted, and
	almost surely, for almost every $u$, $\conti Z_e^{(u)}$ solves \cref{eq:flow_SDE} driven by $\conti W_e$. Moreover, for $u\in(0,1)$, if $\conti Z^\star$ is another $(\mathcal F^{(u)}_t)_t$-adapted solution of \cref{eq:flow_SDE} driven by $\conti W_e$ started at time $u$, then $\conti Z^\star = \conti Z_e^{(u)}$ almost surely.
\end{theorem}

From the above result and the discrete absolutely continuity arguments of \cite{MR3342657,duraj2015invariance}, we can deduce the following analogous result of \cref{thm:coal_con_uncond} (whose proof is omitted). We use the same notation as in \cref{eq:rescaled_version}, and state the result for uniform random times for later convenience.

\begin{theorem}
	\label{thm:discret_coal_conv_to_continuous}
	Let $\bm u_1< \ldots < \bm u_k$ be $k$ sorted independent continuous uniform random variables on $[0,1]$, independent from all other random variables. We have the following convergence in $(\mathcal C([0,1],\R))^{k+2}$: 
	\begin{equation*}
	\left(\bar{\conti W}_n,\bar{\conti Z}^{(\bm u_1)}_n,\ldots,\bar{\conti Z}^{(\bm u_k)}_n\Big|(\bar{\bm W}_t)_{0\leq t\leq {n+1}}  \in \mathcal \Walks^{A,\exc}_{n+2}\right)
	\xrightarrow[n\to\infty]{d}
	\left(\conti W_e,\conti Z_e^{(\bm u_1)},\ldots,\conti Z_e^{(\bm u_k)}\right).
	\end{equation*}
\end{theorem}

\subsection{The construction of the limiting object}\label{sect: constr_limiting_object}

We introduce the limiting \emph{coalescent Baxter permuton}. We place ourselves in the probability space defined above, where $\conti W_e=(\conti X_e, \conti Y_e)$ is a Brownian excursion of correlation $-1/2$ conditioned to stay in the non-negative quadrant. Let
$\conti{Z}_e=\{\conti Z_e^{(u)}\}_{u\in [0,1]}$ be the family of processes given by \cref{thm:ext_and_uni_excursion}, which almost surely solves \cref{eq:flow_SDE} driven by $\conti W_e$ for almost every $u$. From the continuous coalescent-walk process $\conti Z_e$ we build a binary relation $\leq_{\conti Z_e}$ on $[0,1]$ defined as in \cref{eq:coal_to_perm}. Clearly, $(\omega, x,y)\mapsto \idf_{x\leq_{\conti Z_e} y}$ is measurable, and we have the following property whose proof, which relies on path-wise uniqueness, is skipped.
\begin{proposition} \label{prop:total_order}The relation $\leq_{\conti Z_e}$ is a total order on $[0,1]\setminus \bm A$, where $\bm A$ is a random set of zero Lebesgue measure.\end{proposition}
We then define the following random function (note that $(\omega, t)\mapsto \varphi_{\conti Z_e}(t)$ is measurable):
\begin{equation*}
\varphi_{\conti Z_e}(t)\coloneqq\Leb\left( \big\{x\in[0,1]|x \leq_{\conti Z_e} t\big\}\right)=\Leb\left( \big\{x\in[0,t)|\conti Z^{(x)}_{e}(t)<0\big\} \cup \big\{x\in[t,1]|\conti Z^{(t)}_{e}(x)\geq0\big\} \right),
\end{equation*}
where here $\Leb(\cdot)$ denotes the one-dimensional Lebesgue measure.
We define the \emph{coalescent Baxter permuton} as the push-forward of the Lebesgue measure via the map $(\Id,\varphi_{\conti Z_e})$, i.e.
\[\bm \mu_B (\cdot)= \mu_{\conti Z_e}(\cdot)\coloneqq(\Id,\varphi_{\conti Z_e})_{*}\Leb (\cdot)=\Leb\left(\{t\in[0,1]|(t,\varphi_{\conti Z_e}(t))\in \ \cdot\ \}\right). \]
\begin{observation}
We try to give an intuition behind the definition of $\bm \mu_B$. Recall that given a coalescent-walk process $Z=\{Z^{(t)}\}_{t\in[n]}\in\mathcal{C}$, we can associate to it the corresponding Baxter permutation $\sigma=\cpbp(Z)$ and the total order $\leq_{Z}$ on $[n]$. The permutation $\sigma$ satisfies the following property: for every $i\in [n]$, $\sigma(i)=|\{j\in[n]|j\leq_Z i\}|.$
The function $\varphi_{\conti Z_e}$ is a continuous analogue of the permutation $\sigma$, when we consider the continuous coalescent-walk process $\conti Z_e$ instead of a discrete one, and $\mu_{\conti Z_e}$ is the associated permuton.
\end{observation}

The following result is proved as \cite[Proposition 3.1]{maazoun}, relying on \cref{prop:total_order}.
\begin{proposition}\label{prop:order_preserved}
	Almost surely, $\mu_{\conti Z_e}$ is a permuton. 
\end{proposition}
The final proof of \cref{thm:baxter_permuton_conv}, i.e.\ the convergence of uniform Baxter permutations to $\bm \mu_B$, can be found in \cref{sect:perm_conv}. We give here a short sketch. The proof is based on the analysis of pattern extraction from uniform Baxter permutations. \cref{cor:patterns} relates the probability of extracting a specific pattern to the probability that some trajectories of the corresponding coalescent-walk process have given signs at given times. Then, by \cref{thm:discret_coal_conv_to_continuous}, the latter converges to the analogue probability for the limiting continuous coalescent-walk process.

\appendix

\section{The proof of \cref{thm:coal_con_uncond}}\label{sect:coal_con_uncond}

Recall that $\bar{\bm W} = (\bar{\bm X},\bar{\bm Y}) =(\bar{\bm X}_k,\bar{\bm Y}_k)_{k\in \Z}$ is the random plane walk defined below \cref{eq:walk_distrib}, and $\bar{\bm Z} = \wcp(\bar{\bm W})$ is the corresponding coalescent-walk process. We need the following result whose proof is left to the complete version of this extended abstract.

\begin{proposition}\label{prop:trajectories_are_rw}
	For every $u\in \Z$, $\bar{\bm Z}^{(u)}$ has the distribution of a random walk with the same step distribution as $\bar{\bm Y}$ (which is the same as that of $-\bar{\bm X}$).
\end{proposition}

\begin{remark}
	Recall that the increments of a walk of a coalescent-walk process are not always equal to one of the increments of the corresponding walk (see for instance \cref{eq:increments}). The statement of \cref{prop:trajectories_are_rw} is a sort of ``miracle'' of the geometric distribution.
\end{remark}

\begin{proof}[Proof of \cref{thm:coal_con_uncond}]
	The first step in the proof is to establish convergence of the components of the vector on the left-hand side of \cref{thm:coal_con_uncond}.
	By a classical invariance principle, we get that $\bar{\conti W}_n=(\bar{\conti X}_n,\bar{\conti Y}_n)$ converges to $\conti W=(\conti X, \conti Y)$ in distribution. Using \cref{prop:trajectories_are_rw}, and applying again the invariance principle, we get that $(\bar{\conti Z}^{(u)}_n({u+t}))_{t\geq 0}$, converges to a one-dimensional Brownian motion. This gives the marginal convergence thanks to \cref{rem:solution_of_SDE_are_BM}.
	
	\medskip
	
	The second step in the proof is to establish joint convergence. Marginal convergence gives joint tightness, so that by Prokhorov's theorem, to show convergence, one only needs to identify the distribution of all joint subsequential limits. Assume that along a subsequence, we have 
		\begin{equation*}
		\left(\bar{\conti W}_n,\bar{\conti Z}^{( u_1)}_n,\ldots,\bar{\conti Z}^{(u_k)}_n\right) 
		\xrightarrow[n\to\infty]{d}
		\left(\conti W,\tilde {\conti Z_1},\ldots,\tilde {\conti Z_k}\right).
		\end{equation*}
	Using Skorokhod's theorem, we may define all involved variables on the same probability space and assume that the convergence is almost sure. The joint distribution of the right-hand-side is unknown for now, but we will show that for every $1\leq i \leq k$, $\tilde {\conti Z_i} = \conti Z^{(u_i)}$ a.s., which would complete the proof. Recall that $\conti Z^{(u_i)}$ is the strong solution of \cref{eq:SDE}, started at time $u_i$ and driven by $\conti W =(\conti X,\conti Y)$, which exists thanks to \cref{thm:ext_and_uni}.	
	Let us now fix $i$ and abbreviate $u=u_i$,  $\tilde {\conti Z} = \tilde {\conti Z_i}$. Our goal is to show that $\tilde {\conti Z}$ also verifies \cref{eq:SDE} and apply path-wise uniqueness.	
		
	Let $\mathcal F_t = \sigma(\conti W(s),\tilde{\conti Z}(s), s\leq t)$. This gives a filtration for which $\conti W$ and $\tilde{\conti Z}$ are adapted.
	We will show that $\conti W$ is an $(\mathcal F_t)_t$-Brownian motion, that is for $t\in \R, s\geq 0$, 
	$(\conti W(t+s)-\conti W(t)) \indep \mathcal F_t$. 
	For fixed $n$, by definition of a random walk, $\bar{\conti W}_n({t+s}) - \bar{\conti W}_n(t)$ is independent from $\sigma(\bar{\bm W}_k, k\leq \lfloor n t \rfloor)$. Therefore, by the definition given in \cref{eq:increments},
	\begin{equation}
	\left(\bar{\conti W}_n({t+s})-\bar{\conti W}_n(t)\right) \ \indep\ \left(\bar{\conti W}_n(r),\bar{\conti Z}^{(u)}_n(r)\right)_{r\leq n^{-1}\lfloor nt \rfloor}.
	\end{equation}
	By convergence, we obtain that $\conti W(t+s)-\conti W(t)$ is independent from $\left(\conti W(r),\tilde{\conti Z}(r)\right)_{r\leq t}$, completing the claim that $\conti W$ is an $(\mathcal F_t)_t$-Brownian motion.
	
	Now fix a rational $\eps>0$ and a rational $t>u$ such that $\tilde{\conti Z}(t)>\eps$. There is $\delta>0$ so that $\tilde{\conti Z}>\eps/2$ on $[t-\delta,t+\delta]$. By almost sure convergence, there is $N_0$ such that for $n\geq N_0$, $\bar{\conti Z}^{(u)}_n>\eps/4$ on $[t-\delta,t+\delta]$. 
	On this interval, outside of the event 
	\[\{\sup_{1\leq i  \leq n}  |\bar{\conti Y}_{i} - \bar{\conti Y}_{i-1}|\geq \sqrt{2n}\eps/4\},\]
	$\bar{\conti Z}^{(u)}_n - \bar{\conti Y}_n$ is constant  by construction of the coalescent-walk process. As a result (the probability of the bad event is bounded by $Ce^{-c\sqrt n}$ ), the limit $\tilde{\conti Z} - \conti Y$ is constant too almost surely. We have shown that almost surely $\tilde{\conti Z} - \conti Y$ is locally constant on $\{t:\tilde{\conti Z}(t)>\eps\}$. This translates into the following equality:
	\[\int_{u}^t \idf_{\{\tilde{\conti Z}(r)>\eps\}} d\tilde{\conti Z}(r) = \int_u^t \idf_{\{\tilde{\conti Z}(r)>\eps\}} d\conti Y(r). \]
	The stochastic integrals are well-defined: on the left-hand side by considering the canonical filtration of $\tilde {\conti Z}$, on the right-hand-side by considering $(\mathcal F_t)_t$.	
	The same can be done for negative values, leading to 
	\[\int_u^t \idf_{\{|\tilde{\conti Z}(r)|>\eps\}} d\tilde{\conti Z}(r)  =  \int_{u}^t \idf_{\{\tilde{\conti Z}(r)>\eps\}} d\conti Y(r) -  \int_{u}^t \idf_{\{\tilde{\conti Z}(r)<-\eps\}} d\conti X(r). \]
	
	By stochastic dominated convergence theorem \cite[Thm. IV.2.12]{revuz2013continuous}, one can take the limit as $\eps\to 0$,  and obtain 
	\[\int_u^t \idf_{\{\tilde{\conti Z}(r)\neq0\}}d\tilde{\conti Z}(r)  =  \int_{u}^t \idf_{\{\tilde{\conti Z}(r)>0\}} d\conti Y(r) -  \int_{u}^t \idf_{\{\tilde{\conti Z}(r)<0\}} d\conti X(r). \]
	Thanks to the fact that $\tilde{\conti Z}$ is Brownian, $\int_u^t \idf_{\{\tilde{\conti Z}(r)=0\}}d\tilde{\conti Z}(r) = 0$, so that the left-hand side equals $\tilde {\conti Z}(t)$.
	As a result $\tilde {\conti Z}$ verifies \cref{eq:SDE} and we can apply path-wise uniqueness (\cref{thm:ext_and_uni}) to complete  the proof.
\end{proof}

\section{The proof of \cref{thm:baxter_permuton_conv}}\label{sect:perm_conv}

Recall that permuton convergence has been defined in \cref{defn:perm_conv}. We present one its characterizations (which comes from \cite[Theorem 2.5]{bassino2017universal}), expressed in terms of random induced patterns. For $n\in\Z_{>0}$, we denote by $\Perms_n$ the set of permutations of size $n$. Let $1\leq k\leq n$, $\sigma\in \Perms_n$ and $I = \{i_1,\ldots i_k\}$ with $1\leq i_1<\dots<i_k\leq n$. The pattern in $\sigma$ induced by $I$ is the only permutation $\pi\in\Perms_k$ such that the $k$ values $\sigma(i_1),\dots, \sigma(i_k)$ are order isomorphic to $\pi(1),\dots, \pi(k)$. In this case, we write $\pat_I(\sigma)=\pi$.

We also define permutations induced by $k$ points in the square $[0,1]^2$. Take a sequence of $k$ points $(X,Y)=((x_1,y_1),\dots, (x_k,y_k))$ in $[0,1]^2$ in general position, i.e.\ with distinguished $x$ and $y$ coordinates. 
We denote by $(x_{(1)},y_{(1)}),\dots, (x_{(k)},y_{(k)})$ the $x$-reordering of $(X,Y)$,
i.e.\ the unique reordering of the sequence $((x_1,y_1),\dots, (x_k,y_k))$ such that
$x_{(1)}<\cdots<x_{(k)}$.
Then the values $(y_{(1)},\ldots,y_{(k)})$ are in the same
relative order as the values of a unique permutation, that we call the \emph{permutation induced by} $(X,Y)$.

\begin{proposition}\label{prop:perm_charact}
	 Let $\bm{\sigma}_n$ be a random permutation of size $n$, and $\bm I_n^k=\{\bm i^1_n,\dots,\bm i^k_n\}$ be a uniform $k$-element subset of $[n]$, independent of $\bm\sigma_n$. Let  $\bm \mu$ be a random permuton, and denote $\Perm_{k}(\bm\mu)$ the unique permutation\footnote{Note that if $\mu$ is a permuton, then it has uniform marginals and so the $x$ and $y$ coordinates of $k$ points sampled according to $\mu$ are a.s.\ distinct.} induced by $k$ independent points in $[0,1]^{2}$ with common distribution $\bm \mu$ conditionally\footnote{This is possible by considering the new probability space described in \cite[Section 2.1]{bassino2017universal}.} on $\bm \mu$. Then
	 \[\mu_{\bm{\sigma}_n} \stackrel{d}{\to} \bm{\mu} \iff \forall k \in \Z_{>0},\;\forall\pi\in \Perms_k,\quad \P(\pat_{\bm I_n^k}(\bm \sigma_n)=\pi) \to \Prob(\Perm_k(\bm \mu) = \pi).\]
\end{proposition}

We can now prove \cref{thm:baxter_permuton_conv}. First we state a consequence of the fact that $\mu_{\conti Z_e}$ is a permuton and that $\conti Z_e^{(s)}(t)$ are continuous random variables, which allows us to get rid of equalities:
\begin{lemma}\label{lem:boundary}
	Almost surely, for almost every $s<t\in[0,1]$, 
	$\conti Z_e^{(s)}(t)\neq 0$. Then $\conti Z_e^{(s)}(t)>0$ implies $\varphi_{\conti Z_e}(s)<\varphi_{\conti Z_e}(t)$, and  $\conti Z_e^{(s)}(t)<0$ implies $\varphi_{\conti Z_e}(s)>\varphi_{\conti Z_e}(t)$.
\end{lemma}

\begin{proof}[Proof of \cref{thm:baxter_permuton_conv}]
	We reuse here the notation of \cref{thm:discret_coal_conv_to_continuous}. In particular $\bar{\bm W}$ is a $\nu$-random walk and $\bar{\bm Z} = \wcp(\bar{\bm W})$ is the associated coalescent-walk process. Let $\bm \sigma_n = \cpbp(\bar{\bm Z}|_{[n]})$. Let $\mathcal E_n$ denote the event $\{(\bar{\bm W}_t)_{0\leq t\leq {n+1}}  \in \mathcal \Walks^{A,\exc}_{n+2}\}$. By \cref{prop:unif_law} and the fact that the mapping $\cpbp\circ\wcp$ is a size-preserving bijection, conditioned on $\mathcal E_n$, $\bm \sigma_n$ is a uniform Baxter permutation.

	Fix $k\geq 1$ and $\pi \in \Perms_k$. For $n\geq k$, let $\bm I_n=\{\bm i^1_n,\dots,\bm i^k_n\}$ be a uniform $k$-element subset of $[n]$, independent of $\bm \sigma_n$.
	In view of \cref{prop:perm_charact}, to complete the proof, we will show that 
	\[\Prob(\pat_{\bm I_n}(\bm \sigma_n)=\pi\mid \mathcal E_n) \xrightarrow[n\to\infty]{}\P(\Perm_{k}(\mu_{\conti Z_e}) = \pi).\]
	Thanks to \cref{cor:patterns}, we have
	\[\Prob(\pat_{\bm I_n}(\bm \sigma_n )=\pi\mid \mathcal E_n)
	=\Prob\Big( \forall_{1\leq \ell< s\leq k},\,\bar{\bm Z}^{(\bm i^\ell_n)}_{\bm i^s_n}\geq 0 \iff \pi(\ell)>\pi(s)\mid \mathcal E_n\Big).\]
	Let $(\bm u_1,\dots,\bm u_k)$ be the sorted vector of $k$ independent uniform continuous random variables in $[0,1]$. For every $n\geq 1$, one can couple $\bm I_n$ and $(\bm u_1,\dots,\bm u_k)$ so that $\bm i_n^j = \lfloor n \bm u_j \rfloor$ for every $1\leq j \leq k$, with an error of probability $O(1/n)$. As a result, 

	\begin{align}
	\Prob(\pat_{\bm I_n}(\bm \sigma_n )=\pi\mid \mathcal E_n)=&\P\Big(\forall_{1\leq \ell< s\leq k},\,(2n)^{-1/2}\bar{\bm Z}^{(\bm u_\ell)}_{\bm u_s}\geq 0 \iff \pi(\ell)>\pi(s) \mid \mathcal E_n \Big)+ O(1/n) \nonumber\\
	\xrightarrow[n\to\infty]{}&
	\P\Big(\forall_{1\leq \ell< s\leq k},\,\conti Z^{(\bm u_\ell)}_e({\bm u_s})\geq 0 \iff \pi(\ell)>\pi(s)\Big) \nonumber\\
	=& \Prob\left(\forall_{1\leq \ell< s\leq k}, 
	\begin{cases}\pi(\ell)>\pi(s) \implies \varphi_{\conti Z_e}(\bm u_\ell)>\varphi_{\conti Z_e}(\bm u_s)\\
	\pi(\ell)<\pi(s) \implies \varphi_{\conti Z_e}(\bm u_\ell)<\varphi_{\conti Z_e}(\bm u_s)
	\end{cases}
	 \right), \label{eq:goal_proof}
	\end{align}
	where for the limit we used the convergence in distribution of \cref{thm:discret_coal_conv_to_continuous} together with the Portmanteau theorem. Additionally, \Cref{lem:boundary} is used both to take care of the boundary effect in the Portmanteau theorem, and to do the rewriting in the last line.

	In order to finish the proof, it is enough to check that the probability on the right-hand side of \cref{eq:goal_proof} equals $\P(\Perm_{k}(\mu_{\conti Z_e}) = \pi)$. This is clear since by definition of $\Perm_k$ and $\mu_{\conti Z_e}$,  $\Perm_{k}(\mu_{\conti Z_e})$ is the permutation induced by
	$\big((\bm u_1,\varphi_{\conti Z_e}(\bm u_1)),\dots, (\bm u_k,\varphi_{\conti Z_e}(\bm u_k))\big).$
\end{proof}

\section{Simulations of large Baxter permutations}\label{sect:simulations}

The simulations for Baxter permutations presented in the first page of this extended abstract have been obtained in the following way:
\begin{enumerate}
	\item first, we have sampled a uniform random walk of size $n+2$ in the non-negative quadrant starting at $(0,0)$ and ending at $(0,0)$ with increments distribution given by \cref{eq:walk_distrib}. This has been done using a "rejection algorithm": it is enough to sample a walk $W$ starting at $(0,0)$ with increments distribution given by \cref{eq:walk_distrib}, up to the first time it leaves the non-negative quadrant. Then one has to check if the last step inside the non-negative quadrant is at the origin $(0,0)$. When this is the case (otherwise we resample a new walk), the part of the walk $W$ inside the non-negative quadrant, denoted $\tilde W$, is a uniform walk of size $|\tilde W|$ in the non-negative quadrant starting at $(0,0)$ and ending at $(0,0)$ with increments distribution given by \cref{eq:walk_distrib}.
	\item Removing the first and the last step of $\tilde W$, thanks to \cref{prop:unif_law}, we obtained a uniform random walk in $\mathcal W_n$.
	\item Finally, applying the mapping $\cpbp\circ\wcp$ to the walk given by the previous step, we obtained a uniform Baxter permutation of size $n$ (thanks to \cref{thm:keyres}).
\end{enumerate}
Note that our algorithm gives a uniform Baxter permutation of random size.
\bibliography{bibli}

\begin{thebibliography}{10}

\bibitem{bassino2017universal}
F.~Bassino, M.~Bouvel, V.~F{\'e}ray, L.~Gerin, M.~Maazoun, and A.~Pierrot.
\newblock Universal limits of substitution-closed permutation classes.
\newblock {\em Journal of the European Mathematical Society, to appear}, 2019.

\bibitem{bassino2018brownian}
F.~Bassino, M.~Bouvel, V.~F{\'e}ray, L.~Gerin, and A.~Pierrot.
\newblock The {B}rownian limit of separable permutations.
\newblock {\em The Annals of Probability}, 46(4):2134--2189, 2018.

\bibitem{MR0184217}
G.~Baxter.
\newblock On fixed points of the composite of commuting functions.
\newblock {\em Proc. Amer. Math. Soc.}, 15:851--855, 1964.
\newblock URL: \url{https://doi.org/10.2307/2034894}, \href
  {http://dx.doi.org/10.2307/2034894} {\path{doi:10.2307/2034894}}.

\bibitem{MR1873300}
I.~Benjamini and O.~Schramm.
\newblock Recurrence of distributional limits of finite planar graphs.
\newblock {\em Electron. J. Probab.}, 6:no. 23, 13, 2001.
\newblock URL: \url{https://doi.org/10.1214/EJP.v6-96}, \href
  {http://dx.doi.org/10.1214/EJP.v6-96} {\path{doi:10.1214/EJP.v6-96}}.

\bibitem{MR2734180}
N.~Bonichon, M.~Bousquet-M\'{e}lou, and \'{E}. Fusy.
\newblock Baxter permutations and plane bipolar orientations.
\newblock {\em S\'{e}m. Lothar. Combin.}, 61A:Art. B61Ah, 29, 2009/11.

\bibitem{borga2018local}
J.~Borga.
\newblock {Local convergence for permutations and local limits for uniform
  $\rho $-avoiding permutations with $|\rho |=3$}.
\newblock {\em Probability Theory and Related Fields}, 176(1):449--531, 2020.
\newblock \href {http://dx.doi.org/https://doi.org/10.1007/s00440-019-00922-4}
  {\path{doi:https://doi.org/10.1007/s00440-019-00922-4}}.

\bibitem{borga2018localsubclose}
J.~Borga, M.~Bouvel, V.~F{\'e}ray, and B.~Stufler.
\newblock A decorated tree approach to random permutations in
  substitution-closed classes.
\newblock {\em arXiv preprint:1904.07135}, 2019.

\bibitem{borga2019square}
J.~Borga and E.~Slivken.
\newblock Square permutations are typically rectangular.
\newblock {\em Annals of Applied Probability, to appear}, 2019.

\bibitem{MR3882190}
M.~\c{C}a\u{g}lar, H.~Hajri, and A.~H. Karaku\c{s}.
\newblock Correlated coalescing {B}rownian flows on {$\Bbb R$} and the circle.
\newblock {\em ALEA Lat. Am. J. Probab. Math. Stat.}, 15(2):1447--1464, 2018.

\bibitem{MR3342657}
D.~Denisov and V.~Wachtel.
\newblock Random walks in cones.
\newblock {\em Ann. Probab.}, 43(3):992--1044, 2015.
\newblock URL: \url{https://doi.org/10.1214/13-AOP867}, \href
  {http://dx.doi.org/10.1214/13-AOP867} {\path{doi:10.1214/13-AOP867}}.

\bibitem{MR3238333}
T.~Dokos and I.~Pak.
\newblock The expected shape of random doubly alternating {B}axter
  permutations.
\newblock {\em Online J. Anal. Comb.}, 9:12, 2014.

\bibitem{duraj2015invariance}
J.~Duraj and V.~Wachtel.
\newblock Invariance principles for random walks in cones.
\newblock {\em arXiv preprint:1508.07966}, 2015.

\bibitem{MR2763051}
S.~Felsner, \'{E}. Fusy, M.~Noy, and D.~Orden.
\newblock Bijections for {B}axter families and related objects.
\newblock {\em J. Combin. Theory Ser. A}, 118(3):993--1020, 2011.
\newblock URL: \url{https://doi.org/10.1016/j.jcta.2010.03.017}, \href
  {http://dx.doi.org/10.1016/j.jcta.2010.03.017}
  {\path{doi:10.1016/j.jcta.2010.03.017}}.

\bibitem{GHS}
E.~Gwynne, N.~Holden, and X.~Sun.
\newblock Joint scaling limit of a bipolar-oriented triangulation and its dual
  in the {P}eanosphere sense.
\newblock {\em arXiv preprint:1603.01194}, 2016.

\bibitem{gwynne2017mating}
E.~Gwynne, N.~Holden, and X.~Sun.
\newblock A mating-of-trees approach for graph distances in random planar maps.
\newblock {\em arXiv preprint arXiv:1711.00723}, 2017.

\bibitem{gwynne2019mating}
E.~Gwynne, N.~Holden, and X.~Sun.
\newblock Mating of trees for random planar maps and {L}iouville quantum
  gravity: a survey.
\newblock {\em arXiv preprint:1910.04713}, 2019.

\bibitem{hoffman2019scaling}
C.~Hoffman, D.~Rizzolo, and E.~Slivken.
\newblock Scaling limits of permutations avoiding long decreasing sequences.
\newblock {\em arXiv preprint:1911.04982}, 2019.

\bibitem{hoppen2013limits}
C.~Hoppen, Y.~Kohayakawa, C.~G. Moreira, B.~R{\'a}th, and R.~M. Sampaio.
\newblock Limits of permutation sequences.
\newblock {\em Journal of Combinatorial Theory, Series B}, 103(1):93--113,
  2013.

\bibitem{kenyon2015permutations}
R.~Kenyon, D.~Král', C.~Radin, and P.~Winkler.
\newblock Permutations with fixed pattern densities.
\newblock {\em Random Structures \& Algorithms}, 2019.
\newblock URL: \url{https://onlinelibrary.wiley.com/doi/abs/10.1002/rsa.20882},
  \href
  {http://arxiv.org/abs/https://onlinelibrary.wiley.com/doi/pdf/10.1002/rsa.20882}
  {\path{arXiv:https://onlinelibrary.wiley.com/doi/pdf/10.1002/rsa.20882}},
  \href {http://dx.doi.org/10.1002/rsa.20882} {\path{doi:10.1002/rsa.20882}}.

\bibitem{MR3945746}
R.~Kenyon, J.~Miller, S.~Sheffield, and D.~B. Wilson.
\newblock Bipolar orientations on planar maps and {${\rm SLE}_{12}$}.
\newblock {\em Ann. Probab.}, 47(3):1240--1269, 2019.
\newblock URL: \url{https://doi.org/10.1214/18-AOP1282}, \href
  {http://dx.doi.org/10.1214/18-AOP1282} {\path{doi:10.1214/18-AOP1282}}.

\bibitem{maazoun}
M.~Maazoun.
\newblock On the {B}rownian separable permuton.
\newblock {\em Combinatorics, Probability and Computing}, page 1–26, 2019.
\newblock \href {http://dx.doi.org/10.1017/S0963548319000300}
  {\path{doi:10.1017/S0963548319000300}}.

\bibitem{madras2010random}
N.~Madras and H.~Liu.
\newblock Random pattern-avoiding permutations.
\newblock {\em Algorithmic Probability and Combinatorics, AMS, Providence, RI},
  pages 173--194, 2010.

\bibitem{mp}
S.~Miner and I.~Pak.
\newblock The shape of random pattern-avoiding permutations.
\newblock {\em Adv. in Appl. Math.}, 55:86--130, 2014.
\newblock URL: \url{http://dx.doi.org/10.1016/j.aam.2013.12.004}, \href
  {http://dx.doi.org/10.1016/j.aam.2013.12.004}
  {\path{doi:10.1016/j.aam.2013.12.004}}.

\bibitem{MR3098074}
V.~Prokaj.
\newblock The solution of the perturbed {T}anaka-equation is pathwise unique.
\newblock {\em Ann. Probab.}, 41(3B):2376--2400, 2013.
\newblock URL: \url{https://doi.org/10.1214/11-AOP716}, \href
  {http://dx.doi.org/10.1214/11-AOP716} {\path{doi:10.1214/11-AOP716}}.

\bibitem{revuz2013continuous}
D.~Revuz and M.~Yor.
\newblock {\em Continuous martingales and {B}rownian motion}, volume 293.
\newblock Springer Science \& Business Media, 2013.

\end{thebibliography}

\end{document}